\documentclass[12pt]{article}

\usepackage{lmodern}
\usepackage{setspace}
\usepackage[table]{xcolor}
\usepackage{nicefrac} 
\usepackage{amsmath, amssymb, amsthm, graphicx,bm,mathtools}
\usepackage{cite,url}
\usepackage{caption}
\usepackage{subcaption}
\usepackage{tikz}
\usetikzlibrary{tikzmark,bending}
\usepackage[ total={6.5in, 8.5in}]{geometry}
\usepackage{mathtools}
\usepackage{algorithm}
\usepackage{algpseudocode}
\usetikzlibrary{decorations.pathreplacing,calc}
\usetikzlibrary{arrows.meta}
\newcommand{\ruleset}[1]{\textsc{#1}}

\newcommand{\PN}[2]{\textnormal{PN}\ensuremath{\left(#1,#2\right)}}
\newcommand{\CN}[2]{\textnormal{CN}\ensuremath{\left(#1,#2\right)}}
\newcommand{\SN}[2]{\textnormal{SN}\ensuremath{\left(#1,#2\right)}}
\newcommand{\NN}[2]{\textnormal{NN}\ensuremath{\left(#1,#2\right)}}
\newcommand{\NNg}[3]
{\textnormal{NN}\ensuremath{\left(#1,#2,#3\right)}}
\newcommand{\IRP}{Invariance Reduction Process}
\newcommand{\defeq}{\overset{\mathrm{def}}{=\joinrel=}}
\DeclareMathOperator{\im}{im}

\def\N{\ensuremath{\mathcal{N}}}

\def\P{\ensuremath{\mathcal{P}}}

\def\A{\ensuremath{\mathcal{A}}}
\def\p{\ensuremath{\bm{p}}}

\def\z{\ensuremath{\bm{z}}}
\def\r{\ensuremath{\bm{r}}}

\def\Nim{\ruleset{Nim}}
\def\lc{\left\lceil}   
\def\rc{\right\rceil}
\def\lf{\left\lfloor}   
\def\rf{\right\rfloor}
\def\Dd{TwoDelta}

\newtheorem{theorem}{Theorem}[section]

\newtheorem{corollary}[theorem]{Corollary}
\newtheorem{lemma}[theorem]{Lemma}
\newtheorem{definition}[theorem]{Definition}
\newtheorem{remark}[theorem]{Remark}
\newtheorem{example}{Example}
\newtheorem{observation}{Observation}

\newtheorem{open}{Open Problem}

\newcommand{\cref}[1]{Corollary~\textup{\ref{#1}}}

\title{Necklace Games}
\author{Balaji R. Kadam, Silvia Heubach, Matthieu Dufour}

\begin{document}
\maketitle

\abstract{We define and give results on the game \ruleset{NecklaceNim} \NN{n}{k}, which is  \ruleset{PathNim} \PN{n}{k} with an additional move allowed on the end vertices.  This game arises as a sub-game in the context of solving \ruleset{CircularNim} \CN{n}{k}  when $k-2$ consecutive stacks have been depleted, therefore its solution is critical to solving \ruleset{CircularNim}. We solve the infinite families of \NN{n}{k} when play is allowed on at least half the stacks.}

\section{Introduction}\label{sec:intro}

The focus of this paper is the game of \ruleset{NecklaceNim}, which is one of the many variations of the game of \Nim~which started combinatorial game theory. \Nim~was introduced and solved by Bouton~\cite{Bo1901}. It is a game played on $n$ stacks of tokens. A move consists of selecting a single stack, and then removing at least one and as many as all tokens from the selected stack. Generalizations have taken on different directions, either changing the type of move or allowing \Nim~play on multiple stacks which may have a geometric arrangement. An example of the former is  \ruleset{Whythoff's~Nim}~\cite{Wy1907} played on two stacks, where one can either make a \Nim~move on a single stack or remove the same number of tokens from both stacks. The second type of games can be described as \ruleset{SimplicialNim} games, introduced in~\cite{ES1996}. We rephrase these games as \ruleset{SetNim} games because it simplifies the notation and  avoids a lot of algebraic definitions that are not needed in our context. Furthermore, as outlined in~\cite{KaDuHe2025}, the notion of invariance differs from the notion of circuits, and our focus will be on invariance. 

\begin{definition}
 The game \ruleset{SetNim} \SN{n}{\A} is played on a set of $n$ stacks of tokens which are placed on the vertices $V$ of a graph.\footnote{Vertices may be labeled as $1,\dots,n$, as $0,1,\dots, n-1$, or as  $a,b,c,\dots$, depending on what is convenient\\ to describe the move set. The labeling usually mirrors the way we describe the positions of the game.} The collection of allowed move sets $\A=\{A_1,A_2,\ldots,A_{\ell}\}$ with $\cup_{i=1}^{\ell} A_i =V$ defines the allowed moves, where each set $A_i$ in the collection represents a set of vertices that the player can play on. Playing on such a set means to take (in total) at least one token and as many as all tokens from the stacks of the particular set. We will only list the maximal allowable move sets. The last player to move wins. 
\end{definition}

The unique terminal position in \SN{n}{\A} is the position where all stacks are zero, which we will denote by $\bm{0}$. We start by giving some examples of known games that can be described as a \ruleset{SetNim}~game. Depending on the game, positions may be determined only up to symmetries such as rotations or reflections. 

\begin{example} \hfill
\begin{itemize}
    \item[-] \Nim~on $n$ stacks can be expressed as the game \SN{n}{\A_1}~where $\A_1$ is the set of singletons, $\{\{1\},\ldots,\{n\}\}$.
    \item[-]  \ruleset{Moore's $k$-Nim}~\textnormal{\cite{Mo1910}}~on $n$ stacks can be described as \SN{n}{\A_2}~where $\A_2$ is the collection of all $k$-element subsets of $\{1, \ldots,n\}$. 
    \item[-]   \ruleset{CircularNim} \CN{n}{k}~\textnormal{\cite{DuHe2013}}~is played on  $k$ consecutive vertices of a cycle graph with $n$ nodes. Thus, \CN{n}{k}~equals \SN{n}{\A_3}~with $$\A_3=\{\{i,i+1,\ldots, i+k-1\} \mid  i=0,\ldots, n-1\},$$ where the vertex labels in $\A_3$ are understood to be modulo $n$.
    \item[-] \ruleset{PathNim} \PN{n}{k}~\textnormal{\cite{KaDuHe2025}}~is played on $k$ consecutive vertices of a path with $n$ vertices. Thus \PN{n}{k} corresponds to \ruleset{SetNim}~\SN{n}{\A_4}~with 
$$\A_4=\{\{i,i+1,\ldots, i+k-1\} \mid  i=1,\ldots, n-k+1\}. $$
\end{itemize}
\end{example}

Motivated by sub-games that arise when solving \ruleset{CircularNim} games, for example, game $H$ in~\cite{KaDuHe2025} and Example 3.3 of~\cite{ES1996}, we introduce a new family of games that are of interest in their own right. 
These games correspond to  \ruleset{PathNim} games with an added move on the two end vertices, creating a necklace out of a string of ``pearls". 

\begin{definition}
The game \ruleset{NecklaceNim} \NN{n}{k} is played on stacks of tokens arranged on a path with $n \ge 2$ vertices. A move consists of playing either on $k\geq 2$ consecutive stacks or on the two end stacks, removing at least one token in total from the selected stacks. The last player to move wins. Thus,  \NN{n}{k} is  \ruleset{SetNim}~$\SN{n}{\mathcal{A}}$ with
\[
\mathcal{A} = 
   \bigl\{ \{i,i+1,\ldots,i+k-1\} \;\bigm|\; i=1,\ldots,n-k+1 \bigr\} 
   \cup \{\{1,n\}\}.
\]
\end{definition}

We will give results on the \ruleset{NecklaceNim} games listed in Figure~\ref{fig:necklace_sol}. Combinations of $n$ and $k$ shown in white are still unsolved. Note that \NN{n}{2}~has the same move sets as \CN{n}{2}, which is solved for $n \leq 5$ (see \cite{DuHe2013}).
\begin{figure}[!ht]
\centering

\begin{tikzpicture}[scale=0.85]

\def\c{1}

\foreach \i in {1,...,9} {       
  \foreach \j in {1,...,\i} {

    \ifnum\i=\j
      \fill[green!35] (\j*\c,-\i*\c) rectangle ++(\c,\c);
    \fi

    \ifnum\i=\numexpr\j+1\relax
      \fill[blue!35] (\j*\c,-\i*\c) rectangle ++(\c,\c);
    \fi

    \ifnum\i=\numexpr\j+2\relax
      \fill[red!35] (\j*\c,-\i*\c) rectangle ++(\c,\c);
    \fi

    \ifnum\i=\numexpr\j+3\relax
      \fill[orange!35] (\j*\c,-\i*\c) rectangle ++(\c,\c);
    \fi

    \ifnum\i=\numexpr\j+4\relax
      \fill[cyan!35] (\j*\c,-\i*\c) rectangle ++(\c,\c);
    \fi

    \ifnum\i=\numexpr\j+5\relax
      \fill[purple!35] (\j*\c,-\i*\c) rectangle ++(\c,\c);
    \fi


    \draw (\j*\c,-\i*\c) rectangle ++(\c,\c);
  }
}

\draw (1,-9)--(1,1);
\draw (0,0)--(11,0);
\draw (1,0)--(0,1);

\node at (0.25,0.25) {$n$};
\node at (0.75,0.75) {$k$};

\foreach \j in {2,...,10} {
  \node at (0.5,-\j+1.5) {\j};
}

\foreach \j in {2,...,10} {
  \node at (\j-0.5,0.5) {\j};
}
\fill[white]   (1,-4)    rectangle (2,-3);
\draw (1,-4)    rectangle (2,-3);

\draw[->, line width=1pt, blue]   (1.75,-1.75)--(8.5,-8.5);
\node[rotate=45, scale=0.7] at (1.5,-1.5) {{\footnotesize\textbf{NN(3,2)}}};

\draw[->, line width=1pt, red]    (2.75,-3.75)--(7.5,-8.5);
\node[rotate=45, scale=0.7] at (2.5,-3.5) {{\footnotesize\textbf{NN(5,3)}}};

\draw[->, line width=1pt, orange] (3.75,-5.75)--(6.5,-8.5);
\node[rotate=45, scale=0.7] at (3.5,-5.5) {{\footnotesize\textbf{NN(7,4)}}};

\draw[->, line width=1pt, cyan]   (4.75,-7.75)--(5.5,-8.5);
\node[rotate=45, scale=0.7] at (4.5,-7.5) {{\footnotesize\textbf{NN(9,5)}}};

\begin{scope}[yshift=-0.7cm]

\fill[blue!15]   (7,0)    rectangle (11,-0.7);
\draw           (7,0)    rectangle (11,-0.7);
\node[scale=0.85] at (9,-0.35) {{\footnotesize\textbf{NN(n,n-1)$\cong$NN(3,2)}}};

\fill[red!15]    (7,-0.7) rectangle (11,-1.4);
\draw           (7,-0.7) rectangle (11,-1.4);
\node[scale=0.85] at (9,-1.05) {{\footnotesize\textbf{NN(n,n-2)$\cong$NN(5,3)}}};

\fill[orange!15] (7,-1.4) rectangle (11,-2.1);
\draw           (7,-1.4) rectangle (11,-2.1);
\node[scale=0.85] at (9,-1.75) {{\footnotesize\textbf{NN(n,n-3)$\cong$NN(7,4)}}};

\fill[cyan!15]   (7,-2.1) rectangle (11,-2.8);
\draw           (7,-2.1) rectangle (11,-2.8);
\node[scale=0.85] at (9,-2.45) {{\footnotesize\textbf{NN(n,n-4)$\cong$NN(9,5)}}};

\end{scope}

\fill[white] (1*\c,-5*\c) rectangle ++(\c,\c); 
\draw        (1*\c,-5*\c) rectangle ++(\c,\c);

\fill[white] (1*\c,-6*\c) rectangle ++(\c,\c); 
\draw        (1*\c,-6*\c) rectangle ++(\c,\c);

\fill[white] (2*\c,-6*\c) rectangle ++(\c,\c); 
\draw        (2*\c,-6*\c) rectangle ++(\c,\c);

\fill[white] (2*\c,-7*\c) rectangle ++(\c,\c); 
\draw        (2*\c,-7*\c) rectangle ++(\c,\c);

\fill[white] (3*\c,-8*\c) rectangle ++(\c,\c); 
\draw        (3*\c,-8*\c) rectangle ++(\c,\c);

\draw (4,-9)--(5,-8);
\draw (4,-8.5)--(4.5,-8);
\draw (4,-8.75)--(4.75,-8);
\draw (4,-8.25)--(4.25,-8);
\draw (4.25,-9)--(5,-8.25);
\draw (4.5,-9)--(5,-8.5);
\draw (4.75,-9)--(5,-8.75);

\draw (3,-7)--(4,-6);
\draw (3,-6.5)--(3.5,-6);
\draw (3,-6.75)--(3.75,-6);
\draw (3,-6.25)--(3.25,-6);
\draw (3.25,-7)--(4,-6.25);
\draw (3.5,-7)--(4,-6.5);
\draw (3.75,-7)--(4,-6.75);

\draw (2,-5)--(3,-4);
\draw (2,-4.5)--(2.5,-4);
\draw (2,-4.75)--(2.75,-4);
\draw (2,-4.25)--(2.25,-4);
\draw (2.25,-5)--(3,-4.25);
\draw (2.5,-5)--(3,-4.5);
\draw (2.75,-5)--(3,-4.75);

\draw (1,-3)--(2,-2);
\draw (1,-2.5)--(1.5,-2);
\draw (1,-2.75)--(1.75,-2);
\draw (1,-2.25)--(1.25,-2);
\draw (1.25,-3)--(2,-2.25);
\draw (1.5,-3)--(2,-2.5);
\draw (1.75,-3)--(2,-2.75);

\end{tikzpicture}
   \caption{Solved games are shown in color.  Arrows indicate the games that reduce to the anchor game listed at the top of the arrow. Crosshatched games share the structure of the \P-positions of the associated anchor game.  }
    \label{fig:necklace_sol}
\end{figure}
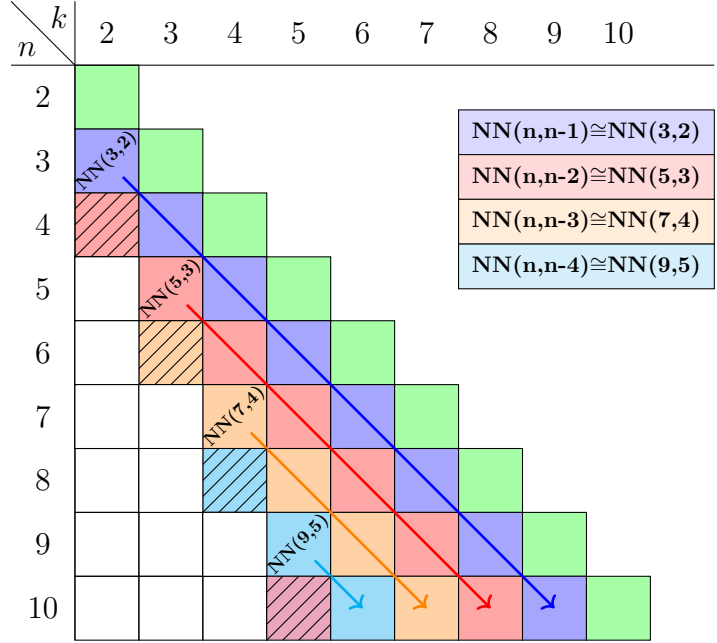

\section{Main Results}\label{sec:results}
We solve \ruleset{NecklaceNim} games when $k \geq n/2$.  There are two components to this solution. The first is that we can reduce infinite families of games to an anchor game, the smallest game in that family that satisfies $k \geq n/2$. The second part of the solution consists of solving those anchor games. 

\begin{theorem} \label{thm:anchor} The game \NN{n}{n-\ell} with $n \ge 4$ and $1\leq \ell < \lfloor n/2 \rfloor$  reduces to \NN{2\ell+1}{\ell+1}.
\end{theorem}

\begin{corollary}\label{cor:k=n-1,2} 
    The set of $\P$-positions of the games \NN{n}{n-1} and  \NN{n}{n-2} are given by 
    \begin{align*}
    \P_{\NN{n}{n-1}}& =\{\p \mid p_1=p_2+\dots+p_{n-1}=p_n\} \textnormal{ for }n \geq 3\\
    \intertext{and}
    \P_{\NN{n}{n-2}}&=\{\p \mid p_1=p_3+\dots+p_{n-1},~p_n=p_2+\dots+p_{n-2}\} \textnormal{ for }n \geq 4,
    \end{align*} 
    respectively.
\end{corollary}

We will refer to the game \NN{2\ell+1}{\ell+1} as the \emph{anchor game} for the infinite family  \NN{n}{n-\ell} with $1 \leq \ell < \lfloor n/2 \rfloor$. This anchor game and the game  \NN{2 \ell}{\ell} share a common structure of the set of \P-positions.

\begin{theorem} \label{thm:redgamen} Let $\p=(p_1,\dots,p_n)$ and $n \geq 3$. Then the games \NN{n}{k} with $k=\lceil n/2 \rceil$, that is the \ruleset{NecklaceNim} games \NN{2\ell}{\ell} and  \NN{2\ell+1}{\ell+1}, have \P-positions given by
\[\P_{n,k}= \{\p\mid  p_1+\dots+ p_{\lf n/2\rf}=p_{\lc n/2\rc+1}+\dots +p_{n}, \,\min( p_1,p_n)=\min_{i=2,\dots,\ell+1}(s_i)\},\]
where $s_i=p_i + p_{i+1} + \dots + p_{i+k-2}$. 
\end{theorem}

\section{Reduction Techniques and Prior Results}\label{sec:tools}
We will use techniques introduced in~\cite{KaDuHe2025} to reduce a specific \ruleset{NecklaceNim} game to a number of smaller, known sub-games and then utilize the results for the sub-games. The first step in this process is to use invariant vectors to create zero stacks and then to eliminate those stacks from consideration. If possible, we will merge additional stacks in the process. The resulting sub-games preserve the type of a position (see Remark~2.3 and Lemma~2.5 in~\cite{KaDuHe2025}). 

Before we go into the details of the reduction process, we will establish the notation that will be used throughout. Depending on what is more convenient to cover the cases for \NN{2\ell+1}{\ell+1}~and  \NN{2\ell}{\ell}~together, we refer to a position $\p$ by using either the generic $\p=(p_1,\dots, p_n)$ or $\p=(a_1,\dots,a_{\ell},(c),b_1,\dots,b_{\ell})$, where the \emph{center stack} $c$ only occurs when $n$ is odd.    Let $$A=\sum_{i=1}^{\ell} a_i=\sum_{i=1}^{\lfloor n/2\rfloor} p_i, \quad B=\sum_{i=1}^{\ell} b_i=\sum_{i=\lc n/2 \rc +1}^n p_i,$$
 $$ m=\min(p_1,p_n)=\min(a_1,b_{\ell}),\quad s^*=\min_{i=2,\dots,\ell+1}(s_i),$$ 
 where $s_i=p_i + p_{i+1} + \dots + p_{i+k-2}$. We use  $A',B',m'$ and $(s^*)'$ to denote the corresponding quantities for the option $\p'$.  We refer to the stacks $a_1,\dots,a_\ell$ as the \emph{$A$-side of the position}, the stacks $b_1,\dots,b_{\ell}$ as the \emph{$B$-side}, and the stacks $p_t,\dots,p_{t+k-2}$ as the \emph{minimum window}, where $t$ is  the smallest index for which $s_t=s^*$. 
 We will refer to $a_1$ and $b_{\ell}$ as the \emph{end stacks}. In addition, let $S_{\ell}$ be the set of positions that satisfy the following two conditions:
$$\quad \textnormal{\emph{sum equation} (SE):} \quad A = B \quad\quad \textnormal{ \emph{minimum equation} (ME):} \quad m=s^*.$$
For ease of readability, we will use the notation $G\defeq \A$ as a shorthand for ``$G$ is the game \SN{n}{\A} with allowed move sets $\A$'' and $G\cong \tilde{G}$ if the games $G$ and $\tilde{G}$ are isomorphic, that is, they have the same move sets. \\

In the proof of our main theorem, \ruleset{PathNim}~games will play a critical role because the relevant \ruleset{NecklaceNim}~games either reduce to smaller \ruleset{NecklaceNim}~games or to  \ruleset{PathNim}~games. We list the result on the \P-positions of \ruleset{PathNim}~games  here for easy reference. 

\begin{theorem}[Theorem 4.1~\cite{KaDuHe2025}] \label{thm:P-pos path} If $k \geq \lceil \frac{n}{2}\rceil$, that is, when play is allowed on at least half of the stacks, then  
\[
\P_{n,k}=\{(a_1,\ldots,a_{u},\underbrace{0,\ldots,0}_{k-1},b_1,\ldots,b_{v})\ |  \sum_{i=1}^{u} a_i=\sum_{j=1}^v b_j, \min\{u,v\}\geq 1\}.
\]
In particular, $\P_{n,n}=\{(0,\ldots,0)\}$. 
\end{theorem}

We will also use that the \P-positions of \CN{3}{2} are given by positions with equal stack heights and those of \CN{4}{2} are given by positions of the form $(a,b,a,b)$ (see~\cite{DuHe2013}, Theorems~2.1 and 2.2).

\subsection{Invariance Reduction}

Invariant vectors and the invariance reduction process will play a crucial role in solving \ruleset{NecklaceNim} games. They are the first step in reducing a game to a smaller known sub-game. 

\begin{definition} \label{def:inv} A vector is called a \emph{zero-one} vector if all its entries are either zeros or ones. A zero-one vector ${\bm z}$ is called \emph{invariant for a set $S$} of game positions if for all $\p \in S$ and any integer $c_{\z}$ for which $\p+c_{\z}\cdot \z \ge \bm{0}$, $\p\in S$ if and only if  $\p+c_{\z}\cdot \z  \in S$. 
\end{definition}

\begin{observation} \label{obs:invodd}
Let $\p$ be a position in set  $S_{\ell}$ of the game \NN{2\ell+1}{\ell+1}. The symmetric vector $\z_{\ell+1}=(1,0,\dots,0,1,0,\dots,0,1)$ is invariant for $S_{\ell}$ because subtracting a multiple $c_{\ell+1}$ of $\z_{\ell+1}$ from $\p$ reduces both $A$ and $B$ by $c_{\ell+1}$ via stacks $a_1$ and $b_{\ell}$, while reduction of the center stack does not affect the two sums, so (SE) remains valid. Since the center stack is a summand in each of the $s_i$, the minimum $s^*$ is reduced by $c_{\ell+1}$, matching the reduction in the minimum $m=\min(a_1,b_{\ell})$. 

\end{observation}

This is not the only invariant vector for the anchor game \NN{2\ell+1}{\ell+1}, but it is the only one specific to this game and the one we will use for the reduction in Section~\ref{sec:proofs}. All vectors shown to be invariant for game \NN{2\ell}{\ell} in Observation~\ref{obs:inveven} are also invariant for \NN{2\ell+1}{\ell+1}.

\begin{observation} \label{obs:inveven}
Let $\p$ be a position in set  $S_{\ell}$ of the game \NN{2\ell}{\ell}. We claim that the  vectors  $\z_i$ with $i=2,\dots,\ell$ that have 1s  at locations $a_1$, $a_i$, $b_{i-1}$, and $b_{\ell}$ are invariant for $S_{\ell}$. Subtraction of a multiple $c_i$ of vector $\z_i$ from $\p \in S_\ell$ reduces both $a_1$ and $b_{\ell}$, as well as one additional stack each on the $A$-side and the $B$-side by $c_i$,  thereby maintaining (SE). A reduction in stack $a_i$ reduces all the $s_j$ with $2 \leq j \leq i$, while the reduction in $b_{i-1}$ (by the same amount) affects $s_j$ for $i+1 \leq j\leq \ell+1$. Thus, every $s_j$ and hence $s^*$ is reduced by that amount as well. This reduction matches the reduction in the end stacks $a_1$ and $b_{\ell}$, maintaining (ME), which implies that $\p \in S_{\ell}$ if and only if $\p-c_i \cdot \z_i \in S_{\ell}$. 
\end{observation}

By Definition~\ref{def:inv}, if $\z$ is invariant for $S_{\ell}$, then $\p$ and $\p-c_{\z}\cdot \z$ have the same outcome class. Since positions with zero stacks are typically easier to analyze, we select coefficients $c_{\z}$ such that $\p-c_{\z}\cdot \z$ has at least one zero.

\subsection{Zero and Merge Reduction}

There are two main ways to reduce games with a larger number of stacks to games with a smaller number of stacks,  the \emph{zero reduction} and the  \emph{merge reduction}. A zero reduction can be used when, during game play or as a result of invariance reduction, one or more stacks have been reduced to zero. These stacks no longer impact the game, so we remove them from the sets $A_i$.

\begin{definition} [Zero reduction] Let $G=$ \SN{n}{\A} be a \ruleset{SetNim}~game on an underlying graph with vertex set $V.$ Let $Z=\{i \mid p_i=0\}$. Then the zero reduction sub-game $\tilde{G}$ is played on the stacks $\tilde{V}=V-Z$, and $\tilde{A_i}={A_i} \cap \tilde{V}$.
\end{definition}

As there is a one-to-one correspondence between $\p\in G$ and $\tilde{\p}\in \tilde{G}$ because $\tilde{p}_i=p_i$ for $i \notin Z$ and $p_i=0$ for $i\in Z$, the  two games have isomorphic game trees and therefore, $\p$ and its reduction $\tilde{\p}$ have the same outcome class. \\

A merge reduction arises when there is a subset of vertices $C$ that has the property that if a move is allowed on one of those stacks, then one can also play on all of the other stacks of $C$. In this case, we can treat this collection of stacks as if it was a single stack for which the number of tokens is the sum of the tokens on the vertices of $C$.

\begin{definition} [Merge reduction] Let $G=$ \SN{n}{\A} be a \ruleset{SetNim}~game on an underlying graph with vertex set $V.$ Assume that a collection of stacks $C\subset V$ has the  property that for any set $A_i \in \A$, either $C\subseteq A_i$ or $C \cap A_i = \emptyset $. Then the merge reduction sub-game $\hat{G}$ is played on the vertices $\hat{V}=(V-C)\cup v^*$ (vertices of $C$ have been merged into a new vertex $v^*$) and the move sets  $\hat{A}_i$ are defined by
$$\hat{A_i}=\left\{\begin{tabular}{ll}
    $A_i$ & \text{if } $A_i\cap C=\emptyset$ \\
  $( A_i-C)\cup v^* $ & \text{otherwise }
\end{tabular}\right.$$
for $A_i \in \A$.
If $\p$ is a position of $G$, then the associated position $\hat{\p}$ of $\hat{G}$ is defined as $\hat{p}_{v^*}=\sum_{i\in C}p_i$, and $\hat{p}_i=p_i$ for  $i \notin C$.
 \end{definition}

Note that while there is a single reduced position $\hat{\p}$ for every $\p\in G$, there may be multiple positions $\p\in G$ that have the same reduced position. We will call the games $G$ and $\hat{G}$ \emph{congruent} as their game outcomes are the same (see Lemma 2.5 in~\cite{KaDuHe2025}), with the merge-reduced game playing the role of a canonical form. \\

The following example, which will play an important role in the proof of Theorem~\ref{thm:redgamen},  illustrates a zero reduction followed by a merge reduction. 

\begin{example}\label{ex:bi_red} Let $S$ be the subspace of positions in game \NN{2\ell}{\ell} that have a zero at location $b_i$ for some $2 \leq i < \ell-1$. Note that the move sets  $A_j=\{a_j,\ldots,b_{j-1}\}$ fall into three categories with regard to which of the stacks $a_i$, $a_{i+1}$, and $b_i$ belong to $A_j$: 

\begin{itemize}
    \item $b_i$ belongs to $A_j$, but neither $a_i$ nor $a_{i+1} $ do \quad $(j> i+1)$,
    \item $a_{i+1}$ and $b_i$ both belong to $A_j$, but $a_i$ does not \quad $(j=i+1)$,
    \item $a_i$ and $a_{i+1}$ both belong to $A_j$, but $b_i$ does not \quad $(j\le i)$.
    \end{itemize}

Eliminating stack $b_i$ reduces  the move sets $A_j$ with $j \geq i+1$ from $\ell$ to $\ell-1$ stacks.  The move set $\tilde{A}_{i+1}=\{a_{i+1},\dots,b_{i-1}\}$  can now be subsumed into the move set ${A}_{i}=\{a_{i},a_{i+1},\dots,b_{i-1}\}$. This eliminates the only move set that contained stack $a_{i+1}$ without also containing $a_i$. As a result, we can now combine stacks $a_i$ and $a_{i+1}$ in the remaining move sets $A_j$  with $j \leq i$, which reduces these move sets to $\ell-1$ playable stacks. Eliminating $b_i$ and merging $a_i$ and $a_{i+1}$ has reduced the total number of stacks by two, and the resulting move set $\A$ is that of  \NN{2\ell-2}{\ell-1}. That is, playing \NN{2\ell}{\ell} on $S$ is the same as playing \NN{2\ell-2}{\ell-1}. Figure~\ref{fig:redbi_red} illustrates the process for $\ell=5$ and $i=3$.

\begin{figure}[!htb]
\centering 
\begin{tikzpicture}
\node at (-4.75,3.75) {Move sets of NN(10,5):};

\node at (-7,2.8) {$b_5$};
\node at (-6.4,2.8) {$a_1$};

\node at (-6.4,2.4) {$a_1$};
\node at (-5.8,2.4) {$a_2$};
\node at (-5.2,2.4) {$a_3$};
\node at (-4.6,2.4) {$a_4$};
\node at (-4,2.4) {$a_5$};

\node at (-5.8,2) {$a_2$};
\node at (-5.2,2) {$a_3$};
\node at (-4.6,2) {$a_4$};
\node at (-4,2) {$a_5$};
\node at (-3.4,2) {$b_1$};

\node at (-5.2,1.6) {$a_3$};
\node at (-4.6,1.6) {$a_4$};
\node at (-4,1.6) {$a_5$};
\node at (-3.4,1.6) {$b_1$};
\node at (-2.8,1.6) {$b_2$};

\node at (-4.6,1.2) {$a_4$};
\node at (-4,1.2) {$a_5$};
\node at (-3.4,1.2) {$b_1$};
\node at (-2.8,1.2) {$b_2$};
\node at (-2.2,1.2) {$b_3$};

\node at (-4,0.8) {$a_5$};
\node at (-3.4,0.8) {$b_1$};
\node at (-2.8,0.8) {$b_2$};
\node at (-2.2,0.8) {$b_3$};
\node at (-1.6,0.8) {$b_4$};

\node at (-3.4,0.4) {$b_1$};
\node at (-2.8,0.4) {$b_2$};
\node at (-2.2,0.4) {$b_3$};
\node at (-1.6,0.4) {$b_4$};
\node at (-1,0.4) {$b_5$};

\draw[very thick,magenta] (-4.85,1.1) to (-5.95,1.1);
\draw[-{Stealth[length=1.75mm, width=1.75mm]}, very thick,magenta] (-5.95,1.1) to[bend left=60] (-5.5,1.6);

\draw[blue,thick] (-1.9,0.15) rectangle (-2.5,1.45);
\node at (-2.2,-0.25) {\color{blue}\textbf{1: delete $b_3$}};

\draw[green!70!black,thick] (-5.5,2.65) rectangle (-4.3,1.4);
\node at (-3.75,3.0) {\color{green!70!black}\textbf{3: merge: $a_3+a_4=\hat{a}$}};

\draw[magenta, very thick] (-4.9,1.4) rectangle (-2.6,0.95);
\node at (-5.75,0.75) {\color{magenta}\textbf{2: subsume}};

\node at (0,1.5) {$\implies$};

\begin{scope}[shift={(7.5,0)}]
\node at (-5.75,3.75) {Move sets of NN(8,4):};
\node at (-7,2.8) {$b_5$};
\node at (-6.4,2.8) {$a_1$};

\node at (-6.4,2.4) {$a_1$};
\node at (-5.8,2.4) {$a_2$};
\node at (-5.2,2.4) {$\hat{a}$};
\node at (-4.6,2.4) {$a_5$};

\node at (-5.8,2) {$a_2$};
\node at (-5.2,2) {$\hat{a}$};
\node at (-4.6,2) {$a_5$};
\node at (-4,2) {$b_1$};

\node at (-5.2,1.6) {$\hat{a}$};
\node at (-4.6,1.6) {$a_5$};
\node at (-4,1.6) {$b_1$};
\node at (-3.4,1.6) {$b_2$};

\node at (-4.6,0.8) {$a_5$};
\node at (-4,0.8) {$b_1$};
\node at (-3.4,0.8) {$b_2$};
\node at (-2.8,0.8) {$b_4$};

\node at (-4,0.4) {$b_1$};
\node at (-3.4,0.4) {$b_2$};
\node at (-2.8,0.4) {$b_4$};
\node at (-2.2,0.4) {$b_5$};
\end{scope}
\end{tikzpicture}
\caption{The game \NN{2\ell}{\ell} reduces to \NN{2\ell-2}{\ell-1} through zero and merge reduction, illustrated for $\ell=5$.}
\label{fig:redbi_red}
\end{figure}
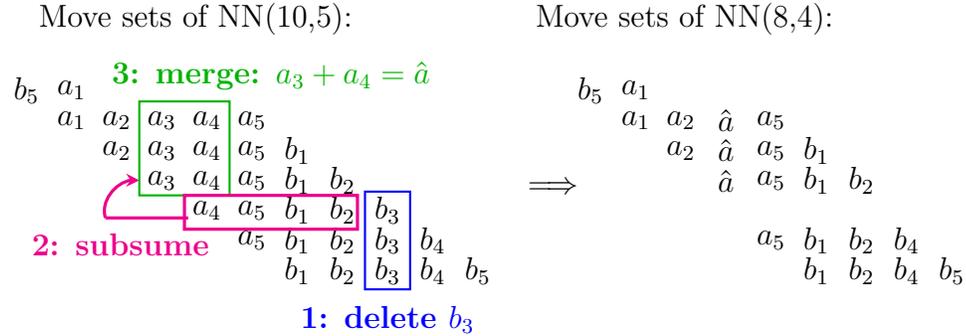
\end{example}

\subsection{\IRP~for Structure}

In~\cite{KaDuHe2025}, the \IRP~was introduced with the goal of finding a move from an \N-position to a \P-position by identifying subspaces that are equivalent to known games, and then using the winning move in the equivalent game to produce a winning move in the game of interest. Here we will modify the \IRP~to allow us to prove the structure of the \P-positions of a game by using the structure of the \P-positions of the smaller games. We adjust steps 3 and 4 of the \IRP~described in~\cite{KaDuHe2025}. 

\vspace{0.2in}

\noindent {\bf \IRP~for Structure}\\
Let $\mathcal{Z}=\{\z_1,\dots,\z_{\ell}\}$ be a set of invariant vectors of the proposed set $\P_G$ and let ${I}_{\z}=\{i \mid z_i=1\}$ be the set of indices at which the invariant vector $\z$ has ones. We refer to the elements of ${I}_{\z}$ as the \emph{$1$-indices of $\z$}; they represent the stacks affected by subtraction of a multiple of $\z$. The maximal multiple of $\z$ that can be subtracted from $\p$ is determined by the minimal stack height among the stacks affected by $\z$, defined as  the \emph{indicator minimum} 
$\im(\z;\p)=\min_{I_{\z}}(p_i)$.  Assume that $\p \in \P_G$.
\begin{enumerate}
    \item {\bf Invariance reduction}: Set $\p^{(0)}=\p$ and iteratively compute  $c_i=\im\left(\z_i;\p^{(i-1)}\right)$ and $\p^{(i)}=\p^{(i-1)}-c_i \z_i$ for $i=1,\ldots,\ell$. At each iteration, either $\p^{(i-1)}$ already has a zero at one of the 1-indices of $\z_i$ (in which case $c_i=0$) or a zero is created at a new location in $\p^{(i)}$. Applying this process, $\tilde{\p}:=\p^{(\ell)}$ has at least one zero. Note that the reduced position $\tilde{\p}$ is still a position in the original game $G$, and invariance implies that  $\tilde{\p}\in \P_G$. 
\item {\bf Subspace game reduction}: For each of the subspaces induced by the invariance reduction, we perform a zero reduction, subsume any move sets that are subsets of larger move sets,  and if possible, perform a merge-reduction to arrive at an isomorphic sub-game $\hat{G}$ with associated position $\hat{\p} \in \hat{G}$.  
\item {\bf Structure verification}: For each reduced sub-game $\hat{G}$, we use the known characterization of its \P-positions together with the defining conditions of the corresponding subspace to verify the structure of the \P-positions of the original game $G$.  
\end{enumerate}

We illustrate the process with two examples that will play a role in the proof of Theorem~\ref{thm:redgamen}. Assume that we apply the \IRP~to the game \NN{2\ell}{\ell} using  the invariant vectors $\z_2,\dots, \z_{\ell}$ with $I_{\z_i}=\{1,i,\ell+i-1,2\ell\}$ (see Observation~\ref{obs:inveven}). We let $\p^{(1)}=\p$ for ease of notation and define  $\Sigma_{min}=\sum_{i=2}^{\ell}\min(a_i,b_{i-1})$ to be the sum of the pair minima. Note that if $m=\min(a_1,b_{\ell})> \Sigma_{min}$, then the \IRP~has $c_i=\min(a_i,b_{i-1})$, resulting in $\tilde{\p}$ with 
\begin{equation}\label{eq:invvals}
  \tilde{a}_1=a_1-\Sigma_{min},\,\, \tilde{a}_i=a_i-c_i, \,\, \,\tilde{b}_{i-1}=b_{i-1}-c_{i},\, \text{ and } \,\tilde{b}_{\ell}=b_{\ell}-\Sigma_{min}   
  \end{equation}  
for $i=2,\dots,\ell$. Not that for each pair $a_i$ and $b_{i-1}$, one of the two stacks has become zero. 

\begin{example}\label{ex:a2-al_red}  We first consider the subspace of positions of the form $\tilde{\p}=(\tilde{a}_1,0,\dots,0,\allowbreak \tilde{b}_1,\allowbreak \dots,\tilde{b}_{\ell})$, with $\ell-1$ consecutive zeros.  
Applying first a zero reduction for stacks $a_2,\dots,a_{\ell}$ we are left with move sets $\{a_1,b_\ell\},\{a_1\}$  and $\{b_1,\dots,b_j\}$ with $1 \leq j\leq \ell$. These can be subsumed into move sets  $\{a_1,b_{\ell}\}$ and $\{b_1,\dots,b_{\ell}\}$, respectively. 
Finally we can merge  stacks $b_1,\dots,b_{\ell-1}$ into $\hat{b}$, giving move sets $\A=\{\{a_1,b_{\ell}\},\{\hat{b},b_{\ell}\}\}$, which are the move sets of \ruleset{PathNim} \PN{3}{2} with path $a_1$ \textemdash~$b_{\ell}$ \textemdash~$\hat{b}$.  Figure~\ref{fig:red_a2-al} illustrates the reduction for $\ell=5$.

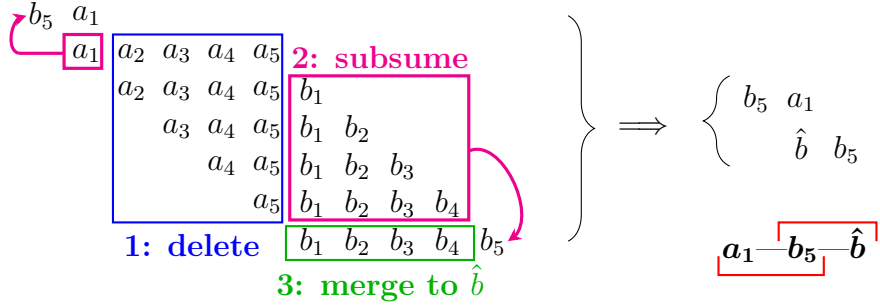
\begin{figure}[!htb]
        \centering 
\begin{tikzpicture}
\node at (-4.75,3.5) {Move sets of NN(10,5):};

\node at (-7,2.8) {$b_5$};
\node at (-6.4,2.8) {$a_1$};

\node at (-6.4,2.3) {$a_1$};
\node at (-5.8,2.3) {$a_2$};
\node at (-5.2,2.3) {$a_3$};
\node at (-4.6,2.3) {$a_4$};
\node at (-4,2.3) {$a_5$};

\node at (-5.8,1.8) {$a_2$};
\node at (-5.2,1.8) {$a_3$};
\node at (-4.6,1.8) {$a_4$};
\node at (-4,1.8) {$a_5$};
\node at (-3.4,1.8) {$b_1$};

\node at (-5.2,1.3) {$a_3$};
\node at (-4.6,1.3) {$a_4$};
\node at (-4,1.3) {$a_5$};
\node at (-3.4,1.3) {$b_1$};
\node at (-2.8,1.3) {$b_2$};

\node at (-4.6,0.8) {$a_4$};
\node at (-4,0.8) {$a_5$};
\node at (-3.4,0.8) {$b_1$};
\node at (-2.8,0.8) {$b_2$};
\node at (-2.2,0.8) {$b_3$};

\node at (-4,0.3) {$a_5$};
\node at (-3.4,0.3) {$b_1$};
\node at (-2.8,0.3) {$b_2$};
\node at (-2.2,0.3) {$b_3$};
\node at (-1.6,0.3) {$b_4$};

\node at (-3.4,-0.2) {$b_1$};
\node at (-2.8,-0.2) {$b_2$};
\node at (-2.2,-0.2) {$b_3$};
\node at (-1.6,-0.2) {$b_4$};
\node at (-1,-0.2) {$b_5$};

\draw[blue,thick] (-6.05,2.55) rectangle (-3.8,0.05);
\node at (-5,-0.25) {\color{blue}\textbf{1: delete}};

\draw[magenta, very thick] (-6.7,2.55) rectangle (-6.2,2.1);

\draw[magenta, very thick] (-3.7,0.0825) rectangle (-1.32,2);
\node at (-2.5,2.25) {\color{magenta}\textbf{2: subsume}};

\draw[-{Stealth[length=1.75mm, width=1.75mm]}, very thick,magenta] (-1.32,1) to[bend left=90] (-0.8,-0.2);
\draw[magenta, very thick] (-6.7,2.3) to (-7.3,2.3);
\draw[-{Stealth[length=1.75mm, width=1.75mm]}, very thick,magenta] (-7.3,2.3) to[bend left=90] (-7.2,2.8);

\draw[green!70!black,thick] (-3.75,0) rectangle (-1.25,-0.45);
\node at (-2.5,-0.75) {\color{green!70!black}\textbf{3: merge to $\hat{b}$}};

\usetikzlibrary{decorations.pathreplacing}

\draw[decorate,decoration={brace,amplitude=10pt}]
(0,2.8) -- (0,-0.2);
\node at (1,1.3) {$\implies$};
\draw[decorate,decoration={brace,amplitude=10pt}]
(2.15,0.8) -- (2.15,2);

\node at (3.1,3.5) {Move sets of PN(3,2):};
\node at (2.5,1.7) {$b_5$};
\node at (3.1,1.65) {$a_1$};

\node at (3.1,1.1) {$\hat{b}$};
\node at (3.7,1.025) {$b_5$};

\node at (3,-0.25) {$\boldsymbol{a_1 \text{—} b_5 \text{—} \hat{b}}$};

\draw[red,thick] (3.4,-0.4)--(3.4,-0.65)--(2,-0.65)--(2,-0.4);

\draw[red,thick] (2.8,-0.2)--(2.8,0.05)--(4.1,0.05)--(4.1,-0.2);

\end{tikzpicture}
\caption{Game \NN{2\ell}{\ell} reduces to \PN{3}{2} through zero and merge reduction on the subspace of positions with zeros at $a_2,\dots,a_{\ell}$, shown for $\ell=5$.}
\label{fig:red_a2-al}
\end{figure}

Now we perform the structure verification step. The \P-positions of \PN{3}{2} are of the form $(\hat{p}_1,0,\hat{p}_3)$ with $\hat{p}_1=\hat{p}_3$ by Theorem~\ref{thm:P-pos path}. Since we are considering the subspace where $\tilde{a}_i=0$ for $i=2,\dots,\ell$, we know that $a_i \leq b_{i-1}$, thus $\Sigma_{min}=\sum_{i=2}^{\ell}a_i$ and $c_i=a_i$. Using Equation~\eqref{eq:invvals} we obtain the following conditions:
\begin{equation} \label{eq:P32_1}
    \hat{p}_1=\tilde{p}_1=a_1-\Sigma_{min}=\hat{p}_3=\hat{b}=\sum_{i=1}^{\ell-1}\tilde{b}_i=\sum_{i=1}^{\ell-1}(b_i-a_{i+1})=\sum_{i=1}^{\ell-1}b_i-\Sigma_{min}
\end{equation}
and
\begin{equation} \label{eq:P32_2}
    \tilde{b}_{\ell}=b_{\ell}-\Sigma_{min}=0\quad \text{or equivalently, }\quad \Sigma_{min} =b_\ell.
\end{equation}
Equation~\eqref{eq:P32_1} implies that $a_1=\sum_{i=1}^{\ell-1}b_i$, which when added to~\eqref{eq:P32_2} gives
$$A=a_1+\Sigma_{min}=\sum_{i=1}^{\ell-1}b_i+b_{\ell}=B,$$
so condition (SE) of $S_{\ell}=\P_G$ is satisfied. Next we need to check the (ME) condition, namely that $m=\min(a_1,b_\ell)=\min(s_2,\dots,s_{\ell+1})=s^*$ where $s_j=p_j+\dots+p_{\ell+j-2}$. Note that for $j=2,\dots,\ell$ we have that $s_j =s_{j+1}-(b_{j-1}-a_j)$. In the sub-space under consideration we have $a_i \leq b_{i-1}$ which implies that $s_2\leq s_3 \leq \dots \leq s_{\ell+1}$, so $s^*=s_2=\Sigma_{min}=b_\ell$. Since $b_{\ell}-\Sigma_{min} = 0 \leq \hat{p}_1=a_1-\Sigma_{min}$, we have that $m=b_{\ell}=s^*$, so (ME) is satisfied. Therefore, $\hat{\p} \in \P_{\text{\PN{3}{2}}}$ if and only if $\p \in S_{\ell}$.
\end{example}

Next we look at a slightly different subspace, also with $\ell-1$ consecutive zeros, but this time at $a_3,\dots,b_1$. The methods are very similar, but the reduction leads to a different \ruleset{PathNim} game.

\begin{example}\label{ex:a3-b1}  
We now consider the subspace of positions of the form $\tilde{\p}=(\tilde{a}_1,\tilde{a}_2,\allowbreak 0,\allowbreak\dots,0,\tilde{b}_2,\dots,\tilde{b}_{\ell})$. 
Applying first a zero reduction on stacks $a_3,\dots,b_1$ we are left with move sets $\{a_1,b_\ell\}$, $\{a_1,a_2\} \supset \{a_2\}$ and $\{b_2,\dots,b_j\}$ with $2 \leq j\leq \ell$. The latter ones can be subsumed into the move set $\{b_2,\dots,b_{\ell}\}$.
Merging  stacks $b_2,\dots,b_{\ell-1}$ into $\hat{b}$ gives $\A=\{\{a_1,b_{\ell}\},\allowbreak \{a_1,a_2\}, \allowbreak\{\hat{b},b_{\ell}\}\}$, the move sets of \ruleset{PathNim} \PN{4}{2} with path $\hat{b}$ \textemdash~$b_{\ell}$ \textemdash~$a_1$\textemdash~$a_2$ and positions
\[ \hat{\p}=
(\hat{p}_1,\hat{p}_2,\hat{p}_3,\hat{p}_4)=\left(\,\sum_{i=2}^{\ell-1}(b_i-a_{i+1}),\,b_{\ell}-\Sigma_{min},\, a_1-\Sigma_{min},\,a_2-b_1\right).
\]
By Theorem~\ref{thm:P-pos path}, the \P-positions of \PN{4}{2}  either have $\hat{p}_2=0$ and $\hat{p}_1=\hat{p}_3+\hat{p}_4$ or $\hat{p}_3=0$ and $\hat{p}_1+\hat{p}_2=\hat{p}_4$. We will show the computation for the first case; the second one follows by symmetry, as the path is only determined up to reading direction.

Since we are considering the subspace where $\tilde{a}_i=0$ for $i=3,\dots,\ell$ and $\tilde{b}_1=0$, we know that $a_i \leq b_{i-1}$ for $i=3,\dots,\ell$ and $a_2>b_1$, thus $\Sigma_{min}=b_1+\sum_{i=3}^{\ell}a_i$. Using the form of $\hat{\p}$ gives
\begin{equation} \label{eq:p2=0}
    \hat{p}_2=0 \quad \Leftrightarrow \quad  b_{\ell}=\Sigma_{min}
\end{equation}
and
\begin{align*} \label{eq:p1=p3+p4}
    \hat{p}_1=\hat{p}_3+\hat{p}_4 \quad & \Leftrightarrow \quad \sum_{i=2}^{\ell-1}b_i-\sum_{i=3}^{\ell}a_i=(a_1-\Sigma_{min})+(a_2-b_1)  \nonumber \\
    &\underset{\eqref{eq:p2=0}}{\Leftrightarrow} \quad  \sum_{i=2}^{\ell-1}b_i+b_1+b_{\ell}=a_1+a_2+\sum_{i=3}^{\ell}a_i,
\end{align*}
thus condition (SE) of $S_{\ell}=\P_G$ is satisfied. For the (ME) condition, note that $a_i\leq b_{i-1}$ for $i\geq 3$ gives $s_3 \leq \dots \leq s_{\ell+1}$ as in Example~\ref{ex:a2-al_red}, and $a_2>b_1$ gives that $s_2>s_3$, so $s^*=s_3=\sum_{i=3}^{\ell}a_i+b_1=b_{\ell}$, where the last equality follows from Equation~\eqref{eq:p2=0}. Finally,  $\hat{p}_3 \geq 0$ implies that $a_1\geq \Sigma_{min}=b_{\ell}$, thus $m=\min(a_1,b_{\ell})=b_{\ell}=s^*$.
Therefore, $\hat{\p} \in \P_{\text{\PN{4}{2}}}$ if and only if $\p \in S_{\ell}$.
\end{example}

\section{Proofs of Main Results}\label{sec:proofs}

We now prove Theorem~\ref{thm:anchor}, the reduction to the anchor game,  Corollary~\ref{cor:k=n-1,2}, and Theorem~\ref{thm:redgamen}, the structure of the \P-positions of the anchor game \NN{2\ell+1}{\ell+1}~and the game \NN{2\ell}{\ell}.

\begin{proof}[Proof of Theorem~\ref{thm:anchor}] Let $k=n-\ell$. Then the $k$-element move sets of \NN{n}{k} are  $\{i,i+1,\dots,i+k-1\}$ for $i=1,\dots,\ell+1$. The sets for $i=1$ and $i=\ell+1$ have the vertices $\{\ell+1,\dots,n-\ell\}$ in common, and so do all the other $k$-element move sets. This set has at least two vertices as long as $n-\ell>\ell+1$, which is the case since $ \ell < \lfloor n/2 \rfloor$ by assumption. Thus we can merge the $n-2\ell$ vertices  $\{\ell+1,\dots,n-\ell\}$ into a single vertex, which reduces the total number of vertices to $n-(n-2\ell)+1=2\ell+1$ vertices. The end stack move is not affected by this merge, and each $k$-element move set now has $k-(n-2\ell)+1=\ell+1$ vertices, resulting in the merge-reduced game \NN{2\ell+1}{\ell+1}. \end{proof}

\begin{proof}[Proof of Corollary~\ref{cor:k=n-1,2}]
  We apply Theorem~\ref{thm:anchor}. For $\ell=1$ and $n \geq 4$, \NN{n}{n-1} reduces to \NN{3}{2} $\cong $ \CN{3}{2}. The \P-positions of \CN{3}{2} are those with equal stack heights, which gives the result.  For $\ell=2$ and $n \geq 6$, \NN{n}{n-2} reduces to \NN{5}{3} with positions $\hat{\p}=(p_1, p_2, \hat{p}_3,p_{n-1},p_n)$ where $\hat{p}_3=p_3+\dots+p_{n-2}$. The game \NN{5}{3} is the same as Example 3.3~\cite{ES1996}, which has \P-positions of the form  $\p =(a+b,c,a,b,a+c)$, so $\P_{\text{NN}(n,n-2)}=\{\p \mid p_1=\hat{p}_3+p_{n-1}, \, \,p_n=p_2+\hat{p}_3\}$. Substituting the value of $\hat{p}_3$ gives the result. Note that \NN{4}{2}~$\cong$ \CN{4}{2} and that the \P-positions of \CN{4}{2} fit the pattern, so the result is also true for $n = 4$. 
  \end{proof}

\begin{proof}[Proof of Theorem~\ref{thm:redgamen}]
We start by showing that any move from $\p \in S_\ell$ leads to an option $\p' \notin S_{\ell}$ by considering the possible legal moves. 
\begin{enumerate}
\item  \emph{Endpoint move} \(\{p_1,p_n\}\): If only one of the two end stacks is reduced, or they are reduced by unequal amounts, then (SE) no longer is valid. If they are reduced by the same amount, then the move results in   $m'<m$. However, all \(s_i\) remain unchanged, thus $(s^*)'=s^*=m>m'$  violating (ME) for $\p'$. 
\item  Move on $\{a_1,\dots,a_{\ell},(c)\}$ (\emph{A-side move}) or $\{(c), b_1,\dots,b_{\ell}\}$ (\emph{B-side move}). When $n$ is even, then these moves result in \(A'\neq B'\) because only one side of the sum equation is affected. The same is true for odd values of $n$, unless play is just on  the center stack $c$. Since the center stack $c$ is part of each of the $s_i$, all values $s_i$ get reduced, so $(s^*)'<s^*=m=m'$ and (ME) fails. 

\item Move on $\{p_j,\dots,p_{j+k-1}\}$ with $2\le j\le n-k$ (\emph{middle move}).
Let \(\Delta_A\) and \(\Delta_B\) be the total reductions in \(A\) and \(B\),
respectively.
\begin{itemize}
    \item If \(\Delta_A\neq\Delta_B\), then \(A'\neq B'\) and (SE) fails.
    \item If \(\Delta_A=\Delta_B>0\), then stacks in both halves were reduced. Let $s^*=s_t=p_t+\dots+p_{t+k-2}$. Since the minimum window spans $k-1$ stacks, necessarily one of the stacks in the minimum window must be played on, no matter where the minimum window is located. (Recall that for odd $n$, while there is the additional center stack, we also have that $k=\ell+1$, so the same argument applies.) Thus $(s^*)'\leq s_t'<s^*=m=m'$, so (ME) fails. 
\end{itemize}
\end{enumerate}
In all cases, at least one of (SE) or (ME) fails for $\p'$, so there is no move from $S_{\ell}$ to $S_{\ell}$. 

Next we need to show that we can always find a move from $\p\notin S_{\ell}$ to $\p' \in S_{\ell}$. To do so, we will use induction on $\ell$ to prove that $S_{\ell}$ is the set of \P-positions for the games \NN{2\ell}{\ell} and \NN{2\ell+1}{\ell+1}. We first prove the  case $n=3$ and then the base case for $\ell=2$, that is, for games \NN{4}{2} and \NN{5}{3}. In the induction step we use suitable invariant vectors to reduce the \NN{2\ell+1}{\ell+1} and \NN{2\ell}{\ell} games to known smaller games.

For $n=3$, \NN{3}{2} $\cong$ \CN{3}{2}, with \P-positions of the form $\{\p \mid p_1=p_2=p_3\}$~\cite{DuHe2013}. The sum condition is $p_1=p_3$, and the min condition is $s_2=p_2=\min(p_1,p_3)$, so both are satisfied and the \P-positions of \NN{3}{2} are given by $S_1$. Now we look at $\ell=2$. Game \NN{4}{2} is \CN{4}{2}, whose \P-positions are given by $\{\p \mid p_1=p_3, \,\,p_2=p_4\}$~\cite{DuHe2013}. Clearly, (SE) is satisfied. We also have that $s_i=p_i$ for $i=2,3$, so
$s^*=\min(p_2,p_3)=\min(p_4,p_1)=m$. Thus, the \P-positions of \NN{4}{2} are given by $S_2$.  The game \NN{2\ell+1}{\ell+1} for $\ell=2$ is \NN{5}{3}. This game was solved in Example 3.3~\cite{ES1996}, with \P-positions given by $\{\p \mid \p =(a+b,c,a,b,a+c)\}$. Clearly, the sum condition is satisfied as $p_1+p_2=p_4+p_5$. Since $k=3$, the min condition is given by 
\[m=\min( p_1,p_5)=\min(a+b,a+c)=\min(p_3+p_4, \, p_2+p_3)=\min(s_3,s_2)=s^*,
\]so (ME) is satisfied. 
Therefore, the base case of the induction hypothesis is proved.  Now assume that the \P-positions of games \NN{2j}{j} and \NN{2j+1}{j+1} are given by $S_{j}$ for $j<\ell$.  

We first consider the game \NN{2\ell+1}{\ell+1}. Let $\p=(a_1,\dots,a_{\ell},c,b_1,\dots,b_{\ell}) \in S_{\ell}$. By Observation~\ref{obs:invodd}, the symmetric vector $\z_{\ell+1}=(1,0,\dots,0,1,0,\dots,0,1)$ is invariant. The  positions  $\tilde{\p}=\p-c_{\ell+1}\cdot \z_{\ell+1}$ with $c_{\ell+1}=\min(a_1,c,b_{\ell})$ fall into one of two sub-spaces: 
\begin{itemize}
    \item $m=\min(a_1,b_{\ell}) \leq c$: Then $\tilde{\p}$ has a zero at one of the end stacks, and without loss of generality, we may assume that $m=a_1$ in what follows. Zero reduction results in the \ruleset{PathNim} game \PN{2\ell}{\ell+1}. By Theorem~\ref{thm:P-pos path}, the \P-positions of this game are given by positions $$\hat{\p}=(\underbrace{p_2,\dots,p_{\ell+1}}_{\ell \text{ stacks}}, \underbrace{p_{\ell+2},\dots,p_{2\ell+1}}_{\ell\text{ stacks}})$$ that have a string of at least $\ell$ zeros, with equal sums on both sides of the zeros. This string of zeros  corresponds to one of the $s_i$ in game \NN{2 \ell +1}{\ell+1} for $i=3,\dots,\ell+1$, hence $s^*(\tilde{\p})=0=\tilde{a_1}$, so (ME) holds. The equality of the two sums on either side of the zero string satisfies the sum condition of $S_{\ell}$. Thus $\hat{\p}\in \P_{\PN{2\ell}{\ell+1}}$ if and only if $\p \in S_{\ell}$.
    \item $m > c$: Then $\tilde{\p}$ has a zero at the center stack, and zero reduction results in the game \NN{2\ell}{\ell} with position $\hat{\p}=(a_1,\dots,a_{\ell},b_1,\dots,b_{\ell})$. 
   \end{itemize} 
We now consider game \NN{2\ell}{\ell}. By Observation~\ref{obs:inveven}, the vectors $\z_i$ with index set $\{1,i,\ell+i-1,n\}$ are invariant for $S_\ell$, and we perform an invariance reduction using all of them. The resulting position $\tilde{\p}$ falls into one of the sub-spaces indicated in the table below.    
\begin{table}[!ht]
\centering
\renewcommand{\arraystretch}{1.3} 
\begin{tabular}{|c|ll|c|}
\hline
\rowcolor[HTML]{DAE8FC} 
\multicolumn{1}{|l|}{\cellcolor[HTML]{DAE8FC}\textbf{Case}} & \multicolumn{2}{l|}{\cellcolor[HTML]{DAE8FC}\textbf{Condition(s)}}                                 & \textbf{Reduced Game}                                                 \\ \hline\hline
\textbf{1}                                                  & \multicolumn{2}{l|}{$\tilde{a}_1=0$ or $\tilde{b}_{\ell}=0$}                                    & \PN{2\ell-1}{\ell}   \\ \hline
\textbf{2}                                                  & \multicolumn{2}{l|}{$\tilde{b}_{i}=0$ for some $i\in\{2,\dots,\ell-1\}$}                        & \NN{2\ell-1}{\ell-1}  \\ \hline
{\textbf{3}} &  \multicolumn{1}{l|}{$\tilde{a}_3=\dots= \tilde{a}_{\ell}=0$}&\multicolumn{1}{c|}{$\tilde{a}_2=0$}                     & \PN{3}{2}   \\ \cline{3-4} 
  & \multicolumn{1}{l|}{} & \multicolumn{1}{c|}{$\tilde{b}_1=0$}  & \PN{4}{2}                                          \\ \hline
\end{tabular}
\end{table}

Case 1 occurs when $m=\min(a_1,b_{\ell})\leq \Sigma_{min}=\sum_{i=2}^{\ell}\min(a_i, b_{i-1})$ and the proof follows as in the case when $n=2\ell+1$.  The reduction to game \NN{2\ell-1}{\ell-1}~in Case 2 follows from Example~\ref{ex:bi_red}. What is left to show is that $S_\ell$ is the set of \P-positions. In this case, $\tilde{\p}$ has values given in Equation~\eqref{eq:invvals}. Figure~\ref{fig:bi_red} shows the effects of the zero reduction of $b_i$ and the merge reduction that combines $a_i$ and $a_{i+1}$. Note that we can consider the zero reduction as a merge of stacks $b_{i-1}$ and $b_i$, which aids in understanding how $\hat{\p}$ and $\tilde{\p}$ relate to each other. 

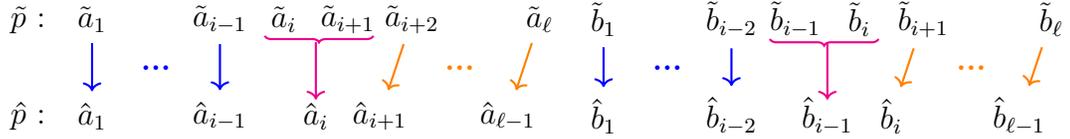
\begin{figure}[!htb]
\begin{center}
\begin{tikzpicture}[scale=0.85]
\node at (0,1.5)   {$\tilde{p}$ : };
\node at (0,0)   {$\hat{p}$ : };
\node (1) at (1,1.5)   {$\tilde{a}_1$};
\node (2) at (3,1.5)   {$\tilde{a}_{i-1}$};
\node (3) at (4,1.5) {$\tilde{a}_i$};
\node (4) at (5,1.5) {$\tilde{a}_{i+1}$};
\node (5) at (6,1.5) {$\tilde{a}_{i+2}$};
\node (6) at (8,1.5)   {$\tilde{a}_{\ell}$};

\node (7) at (9,1.5)   {$\tilde{b}_1$};
\node (8) at (11,1.5)  {$\tilde{b}_{i-2}$};
\node (9) at (12,1.5){$\tilde{b}_{i-1}$};
\node (10) at (13,1.5){$\tilde{b}_{i}$};
\node (11) at (14,1.5){$\tilde{b}_{i+1}$};
\node (12) at (16,1.5)  {$\tilde{b}_{\ell}$};

\draw[
  decorate,
  decoration={brace, mirror},
  magenta,
  thick
] (3.7,1.25) -- (5.4,1.25);

\draw[->,magenta,thick] (4.5,1.15) -- (4.5, 0.25); 

\node (21) at (1,0)   {$\hat{a}_1$};
\node (22) at (3,0)   {$\hat{a}_{i-1}$};
\node (23) at (4.5,0) {$\hat{a}_i$};
\node (24) at (5.5,0) {$\hat{a}_{i+1}$};
\node (25) at (7.5,0)   {$\hat{a}_{\ell-1}$};

\node (26) at (9,0)   {$\hat{b}_1$};
\node (27) at (11,0)  {$\hat{b}_{i-2}$};
\node (28) at (12.5,0){$\hat{b}_{i-1}$};
\node (29) at (13.5,0){$\hat{b}_{i}$};
\node (30) at (15.5,0)  {$\hat{b}_{\ell-1}$};

\draw[
  decorate,
  decoration={brace, mirror},
  magenta,
  thick
] (11.6,1.20) -- (13.3,1.20);

\draw[->,magenta,thick] (12.5,1.15) -- (12.5, 0.25); 

\draw[->,blue,thick] (1) -- (21);        
\draw[->,blue,thick] (2) -- (22); 
\draw[->,blue,thick] (7) -- (26); 
\draw[->,blue,thick] (8) -- (27); 

\draw[->,orange,thick] (5) -- (24); 
\draw[->,orange,thick] (6) -- (25);
\draw[->,orange,thick] (11) -- (29);
\draw[->,orange,thick] (12) -- (30);

\node at (2,0.75)   {{\color{blue}\bf ...}};
\node at (10,0.75)   {{\color{blue}\bf ...}};

\node at (6.75,0.75)   {{\color{orange}\bf ...}};
\node at (14.75,0.75)   {{\color{orange}\bf ...}};
\end{tikzpicture}
\vspace{-17pt}
\end{center}
    \caption{Mapping of $\tilde{\p}$ to $\hat{\p}$ when $\tilde{b}_i=0$. }
    \label{fig:bi_red}
\end{figure}
Specifically, 
\[
\hat{a}_j=\left\{
    \begin{array}{lll}
       \tilde{a}_j  & \text{for}  &  2\leq j<i\\
       \tilde{a}_i +  \tilde{a}_{i+1} & \text{for}  &  j = i\\
       \tilde{a}_{j+1}  & \text{for}  &  i< j \leq \ell-1\\
    \end{array}\right. \quad 
    \hat{b}_j=\left\{
    \begin{array}{lll}
       \tilde{b}_j  & \text{for}  &  1\leq j<i-1\\
       \tilde{b}_{i-1}+   \tilde{b}_{i} & \text{for}  &  j = i-1\\
       \tilde{b}_{j+1}  & \text{for} &  i\leq  j \leq \ell-2.\\
    \end{array}\right.
\]

Now we consider the (ME) equation. Let $\hat{s}_j=\hat{a}_j+\dots+\hat{b}_{j-2}$ for $j=2,\dots,\ell$. When $j \leq i$, then $\hat{s}_j$ contains $\hat{a}_i$ but not $\hat{b}_{i-1}$, and for $j>i$, $\hat{s}_j$ contains $\hat{b}_{i-1}$ but not  $\hat{a}_i$. Thus 
\[
\hat{s}_j=\left\{
    \begin{array}{lll}
       \tilde{a}_j+\dots +\tilde{a}_{\ell}+ \tilde{b}_1+\dots +\tilde{b}_{j-2}=\tilde{s}_j  & \text{for}  &  2\leq j\leq i\\
       \tilde{a}_{j+1}+\dots +\tilde{a}_{\ell}+ \tilde{b}_1+\dots +\tilde{b}_{j-1}=\tilde{s}_{j+1}  & \text{for}  &  i < j\leq \ell.\\
    \end{array}\right. 
\]
Since by assumption $\tilde{b_i}=0$, we have $a_{i+1} \ge b_i$, which implies that $\tilde{s}_{i+1}=\tilde{s}_{i+2}+\tilde{a}_{i+1}-\tilde{b}_{i}\geq \tilde{s}_{i+2}$.   Therefore,
\[
s^*(\tilde{\p})
=\min_{2\le j\le \ell+1} (\tilde{s}_j)
= \min_{\substack{2\le j\le \ell+1\\ j\neq i+1}} (\tilde{s}_j)
=\min_{2 \leq j \leq \ell} (\hat{s}_j)=s^*(\hat{\p}).
\]
Since $\hat{\p} \in S_{\ell-1}$ by induction hypothesis, we have that 
\[
s^*(\hat{\p})=\min(\hat{a}_1,\hat{b}_{\ell-1})=\min(\tilde{a}_1,\tilde{b}_{\ell}).
\]
Thus (ME) holds for $\tilde{\p}$ and because of invariance, also for $\p$. Since the only changes in the zero and merge reductions was to merge two stacks in each half, the sum condition is not affected, so (SE) holds for $\tilde{\p}$ and due to invariance, for $\p$. 

Finally, Case 3 results when none of the $\tilde{b}_i$ of Case 2 are zero, forcing $\tilde{a}_3=\dots \tilde{a}_{\ell}=0$, and then distinguishing whether $a_2 \leq b_1$ and $a_2 >b_1$. The reductions and structure of the \P-positions follow from Example~\ref{ex:a2-al_red} and~Example~\ref{ex:a3-b1}. This completes the proof. 
\end{proof}

\section{How to play \NN{n}{k} with $k=\lceil n/2 \rceil$}
In Section~\ref{sec:proofs}, we proved the pattern for the \P-positions of games \NN{n}{k} with $k=\lceil n/2 \rceil$, but the proof provided little guidance on how to actually find a move from an \N-position to a \P-position. We will provide a ``cheat sheet" to indicate how to move in the various cases when one or both of the (SE) and (ME) conditions are not satisfied. In some cases a simple endpoint move is easily apparent. In instances when an $A$-side move or a middle move are necessary, we will provide algorithms that indicate how to produce the desired move. 

\subsection{Algorithms}

The algorithms presented here are named according to the imbalance they resolve. Before we present the different algorithms, we collect a few facts that will assist us in proving that the algorithms terminate.

\begin{remark} \label{rem:algs}
    Let $A>B$, $\delta=|s^*-m|>0$, $0<m=\min(a_1,b_{\ell})$, and $t$ be the index of the minimum window.
    \begin{enumerate}     
        \item \label{itm:impactind} Play on stack $a_j$ for $j \geq 2$ impacts the values of $s_i$ for $2 \leq i\leq j$ and does not impact the values of $s_i$ for $i>j$. Play on stack $b_y$ affects sums $s_i$ for  $i=y+2,\dots, \ell+1$. The impacted sums are decreased by the same amount as $a_j$ or $b_y$ are. 
        \item \label{itm:maxdedind} When adjusting an individual stack $a_j$ or $b_y$, one cannot reduce the stack by more than the minimal distance of the impacted $s_i$ to the goal $g=\min(m,s^*)$. That is, by at most  
        \begin{itemize}
            \item $m_j=\min_{2 \leq i \leq j}(s_i - g)$ for play on stack $a_j$
            \item $m_y=\min_{y+2 \leq i \leq \ell+1}(s_i - g)$ for play on stack $b_y$. 
        \end{itemize} 
        Reducing stack $a_j$ by $d$ decreases all values $m_i$ for $i \leq j$ by $d$ as well. Likewise,  values $m_i$ for $y+2\leq i \leq \ell+1$ are reduced by $d$ if $b_y$ is reduced by $d$.
        \item \label{itm:m_l+1} When there is overlap between affected sums $s_i$ in simultaneous play on both sides to reduce $\delta=s^*-m$ to zero, then the requirement that no more than $\lfloor{\delta/2}\rfloor$ tokens be removed prevents $\delta$ from falling below zero. Note that by definition, $m_{\ell+1}=\min_{2\leq i \leq \ell+1}(s_i-m)=s^*-m=\delta$. 
        \item \label{itm:consec0} If $\delta=s^*-m>0$, then $m>0$ implies that $s^* >0$. This in turn implies that any sequence of $\ell-1$ consecutive stacks among stacks $a_2$ to $b_{\ell-1}$ cannot consist of only zeros; there is at least one non-zero stack. 
    \end{enumerate}
\end{remark}

The first algorithm deals with positions where both the (SE) and the (ME) conditions are violated, while  the other two algorithms apply when exactly one of the two conditions is not satisfied.

\begin{minipage}{.9\linewidth}
\vspace{-15pt}
\begin{algorithm}[H]
\caption{\Dd-Algorithm}\label{alg:Delta_delta}
\begin{algorithmic}[1]
\Require $\p=(a_1,\dots,a_\ell,b_1,\dots,b_\ell)$ with $A > B$ and $s^*>m$
\State $\r \gets \p$
\State $\Delta \gets A-B$
\State $m \gets \min(a_1,b_{\ell})$
\State Compute the quantities $s_i$ for $i=2,\dots,\ell+1$ 
\State Compute the quantities $m_j=\min_{2 \leq i \leq j}(s_i-m)$ for $j=2,\dots,\ell+1$
\State $j \gets \ell$
\State $\delta\gets m_{\ell+1}$
\For{$j=\ell,\ell-1,\dots,2$}
   \State $d \gets \min(a_j(\p),\Delta, m_j)$
        \If{$d>0$}
            \State $a_j(\r)\gets a_j(\p)-d$
            \State $\Delta\gets \Delta-d$
            \State $m_i\gets m_i-d$ for $2 \leq i \leq j-1$
            \State $\delta \gets \min(\delta,m_j)$     
        \EndIf  
             \If{$\Delta=0$ \textbf{or} $\delta=0$}
                  \State \Return $\r$
            \EndIf         
\EndFor
\end{algorithmic}
\end{algorithm}
\end{minipage}

\vspace{0.005cm}

\begin{lemma} \label{lem:Dd} Let $n=2\cdot \ell \ge 4$, $\p=(a_1,\dots,a_{\ell},\, b_1,\dots,b_{{\ell}})\notin S_{\ell}$ with $A-B=\Delta > 0$ and $m=\min(a_1,b_{\ell})<s^*$.  Then the \Dd-Algorithm returns a position $\r$  with either $\delta=s^*-m=0$ or $\Delta=0$. Specifically, the resulting position $\r$ is of the form $(a_1,\dots,a_{j-1},r_j,0,\dots,0,b_{1},\dots,b_{\ell})$ with $j \geq 2$, that is, play has occurred on exactly $\ell-j+1$ $A$-side stacks. 
\end{lemma}

\begin{proof} The definition of $d$ ensures the move on $a_j$ is legal, that not more than $\Delta$ tokens are removed, and that none of the $s_i$ fall below $m$ by definition of the $m_j$ (see Remark~\ref{rem:algs}.\ref{itm:maxdedind}). Whenever $d>0$, the value of $\Delta$ is reduced, and $\delta$ decreases when the minimum sum $s^*$ is affected by the move.  When the loop has index $j$, one of the following happens:
\begin{itemize}
    \item $d=\Delta$: the algorithm stops with $\Delta(\r)=0$ and $a_j(\r) \geq 0$.
    \item $d=m_j$: the algorithm stops with $\delta(\r)=0$ and $a_j(\r) \geq 0$.
    \item $d=a_j$: then $a_j(\r)=0$.
\end{itemize} 
Thus, if the algorithm were to stop at $j=2$ and neither $\delta=0$ nor $\Delta=0$, then we have $a_2(\r)=\dots = a_{\ell}(\r)=0$, which implies that $s_2=s^*=0<m$, which cannot happen as the definition of $d$ ensures that $s_i \geq m$ in each step.   
\end{proof}

We now illustrate the algorithm with two examples that showcase the two different outcomes of the algorithm. 

\begin{example} \label{ex:DdD1} Let $\p=(4,21,3,2,3,4,2,7,6,5)$ with $A=33$, $B=24$, $\Delta=9$, $m=4$ and $(s_2,s_3,s_4,s_5,s_6)=(29,12, 11, 16, 19)$, which implies $\delta=7$ and $(m_2,m_3,m_4,m_5,m_6)=(25,8, 7, 7, 7)$. The table below shows the initial values in the first row. The first loop starts with $j=5$ and $d=\min(a_5(\p),\Delta,m_5)=\min(3,9,7)=3$, resulting in $a_5(\r)=0,\Delta=6$, $\delta=\min(\delta,m_5)=\min(7,4)=4$, and $\r= (4,21,3,2,0,4,2,7,6,5)$. These  updated values are shown in the row for $j=5$. The table shows the corresponding values for $j=4$ and $j=3$ after which the algorithm ends with $\delta=0$, $\r=(4,21,0,0,0,4,2,7,6,5)$, and $\Delta=1$.

\begin{table}[!htb]
\centering
\begin{tabular}{|l||l||l|l|l|l|l||l||c|}
\hline
\rowcolor[HTML]{DAE8FC} 
$j$ & $\Delta$ & $m_2$ & $m_3$ & $m_4$ & $m_5$ & $m_6$ & $\delta$ & \cellcolor[HTML]{DAE8FC}$\r$ \\ \hline \hline
    & ${\bf 9}$      & $25$  & $8$   & $7$   & $7$   & $7$   & ${\bf 7}$      & $(4,21,3,2,3,4,2,7,6,5)$     \\ \hline
$5$ & ${\bf 6}$      & $22$  & $5$   & $4$   & $4$   &       & ${\bf 4}$      & $(4,21,3,2,{\bf 0},4,2,7,6,5)$     \\ \hline
$4$ & ${\bf 4}$      & $20$  & $3$   & $2$   &       &       & ${\bf 2}$      & $(4,21,3,{\bf 0},0,4,2,7,6,5)$     \\ \hline
$3$ & ${\bf 1}$      & $17$  & $0$   &       &       &       & ${\bf 0}$      & $(4,21,{\bf 0},0,0,4,2,7,6,5)$     \\ \hline
\end{tabular}
\end{table}
\end{example}

\begin{example} \label{ex:Ddd1} Let $\p=(2,15,8,4,5,4,5,5,5,8)$ with $A=34$, $B=27$, $\Delta=7$, $m=2$ and $(s_2,s_3,s_4,s_5,s_6)=(32,21,18,19,19)$,  which implies $\delta=16$ and $(m_2,m_3,m_4,m_5,m_6)=(30,19, 16, 16, 16)$. The table below shows the computations of the \Dd-Algorithm for $j=5$ and $j=4$, ending with $\Delta=0$, $\r=(2,15,8,2,\allowbreak 0, \allowbreak 4, 5, 5, 5,8)$, and $\delta=9$.

\begin{table}[!htb]
\centering
\begin{tabular}{|l||l||l|l|l|l|l||l||c|}
\hline
\rowcolor[HTML]{DAE8FC} 
$j$ & $\Delta$ & $m_2$ & $m_3$ & $m_4$ & $m_5$ & $m_6$ & $\delta$ & \cellcolor[HTML]{DAE8FC}$\r$   \\ \hline
    & ${\bf 7}$      & $30$  & $19$  & $16$  & $16$  & $16$  & ${\bf 16}$     & $(2,15,8,4,5,4,5,5,5,8)$       \\ \hline
$5$ & ${\bf 2}$      & $25$  & $14$  & $11$  & $11$  &       & ${\bf 11}$     & $(2,15,8,4,{\bf 0},4,5,5,5,8)$ \\ \hline
$4$ & ${\bf 0}$      & $23$  & $12$  & $9$   &       &       & ${\bf 9}$      & $(2,15,8,{\bf 2},0,4,5,5,5,8)$ \\ \hline
\end{tabular}
\end{table}
\end{example}
\vspace{-10pt}
We will return to these two examples once we have described the $\delta$- and the $\Delta$-Algorithms.  We start with the $\Delta$-Algorithm.

\begin{minipage}{.9\linewidth}
\begin{algorithm}[H]
\caption{$\Delta$-Algorithm}
\label{alg:Delta}
\begin{algorithmic}[1]
\Require $\p=(a_1,\dots,a_\ell,b_1,\dots,b_\ell)$ with $A>B$  and $m=s^*$
\State $\r \gets \p$;  $\phantom{xx}\Delta \gets A-B$
\State Compute the quantities $s_i$ for $i=2,\dots,\ell+1$ and $s^*=\min(s_i)$
\State $t \gets $ smallest index satisfying $s_t=s^*$
\State Compute the quantities $m_j=\min_{2 \leq i \leq j}(s_i-s^*)$ for $j=2,\dots,t-1$
\State $j \gets t-1$
\While{$\Delta > 0$ \textbf{and} $j \geq 2$}
    \State $d \gets \min(a_j(\p),\Delta,m_j)$
    \If{$d > 0$}
        \State $a_j(\r) \gets a_j(\p)-d$
            \State $\Delta \gets \Delta - d$
        \State $m_i \gets m_i-d$ for $i=2,\dots,j-1$
    \EndIf
    \State $j \gets j - 1$
\EndWhile
\If{$\Delta=0$}
\State \Return $\r$
\ElsIf{$j=2$}
\State $a_1(\r) \gets s_{\ell+1}$
\State \Return $\r$
\EndIf

\end{algorithmic}
\end{algorithm}
\end{minipage}
\vspace{0.5cm}

\begin{lemma} \label{lem:A-side}
Let $n=2\cdot \ell \ge 4$, $\p=(a_1,\dots,a_{\ell},\, b_1,\dots,b_{{\ell}})\notin S_{\ell}$ with $s^* =\min(a_1,b_{\ell})$ and $A-B=\Delta >0$.  Then the $\Delta$-Algorithm returns an option $\r \in S_{\ell}$  that can be reached by an $A$-side move. 
\end{lemma}

\begin{proof}
 The $m_i$ are defined as in Remark~\ref{rem:algs}.\ref{itm:maxdedind} with $g=s^*$.  The definition of $d$ ensures that the move is legal, at most $\Delta$ tokens are removed, and none of the $s_i$  fall below $s^*$.  There are three possible cases for the values of the vector $\r$ at the end of the loop with index $j$: 
\begin{itemize}
    \item $d=a_j(\p)<m_j$: then $s_j(\r) >s^*$ and $a_j(\r)=0$. 
     \item $d=\Delta$: the algorithm stops and returns position $\r \in S_{\ell}$.
     \item $d=m_j\leq a_j(\p)$: then $s_j(\r) =s^*$ and $a_j(\r) \geq 0$.   
\end{itemize} 
If the While loop stops with $\Delta = 0$, then the resulting position satisfies both (SE) and (ME), thus $\r \in S_{\ell}$. Now assume the While loop stops with $j=2$ and $\Delta>0$. If $d=a_2(\p)<m_{2}$, then $s_2(\r)>s^*$ and $a_2(\r)=0$. Then there is a smallest index  $3\leq t' \leq t$ such that $a_{t'}(\r)>0$ (and hence $s_{t'}(\r)=s^*$). If $t'=\ell+1$, then $s_2=0<m$, a contradiction. If $t'\leq \ell$, then
\begin{align*}
    s_2&=a_2+\dots +a_{t'-1}+a_t'+\dots +a_{\ell}\\
    & = a_{t'}+\dots +a_{\ell}+b_1+\dots +b_{t'-2}-(b_1+\dots+b_{t'-2})\\
    & =  s^*-(b_1+\dots+b_{t'-2}) \leq s^*,
\end{align*}
a contradiction to $s_2(\r)>s^*$, so this case cannot happen. 
Finally, if $d=m_2\leq a_2(\p)$, then $s_2(\r)=s^*$ and $A=a_1+s_2=a_1+s^*$. 
If $m=a_1 <b_{\ell}$, then $A<b_\ell+s_{\ell+1}=B$, a contradiction.
 Thus $s^*=m=b_{\ell}$ and $A-B=a_1+s^*-(b_{\ell}-s_{\ell+1})>0$ implies that $a_1>s_{\ell+1}$, so setting $a_1(\r)=s_{\ell+1}$ results in $A=B$ while maintaining $m=\min(a_1(\r),b_{\ell})$. All moves are on the $A$-side, so we have a legal move to $\r \in S_{\ell}$.
\end{proof}

\begin{example} \label{ex:DdD2} We continue Example~\ref{ex:DdD1}. Position $\r=(4,21,0,0,0,4,2,7,6,5)$ with $\Delta=1$ now becomes the input to the  $\Delta$-Algorithm.  The \Dd-Algorithm ended with $m_3=0$, which implies that  $s_3=s^*$ and $t=3$.  Therefore, the $\Delta$-Algorithm starts with $j=2$ and $m_2=17$. We obtain $d=\min(a_2(\p),\Delta,m_2)=\min(21,1,17)=1$, resulting in $a_2(\r)=20$ and $\Delta=0$. The algorithm returns $\p'=(4,20,0,0,0,4,2,7,6,5)\in S_5$.
\end{example}

We now turn to the $\delta$-Algorithm, which fixes the (ME) constraint.

\begin{lemma} \label{lem:d} Let $n=2\cdot \ell \ge 4$, $\p=(a_1,\dots,a_{\ell},\, b_1,\dots,b_{{\ell}})\notin S_{\ell}$ with $A-B=\Delta =0$ and $m=\min(a_1,b_{\ell})<s^*$.  Then the $\delta$-Algorithm returns a position $\r$  with $\delta \in\{0,1\}$. Specifically, the resulting position $\r$ is of the form
$$(a_1(\p),\dots,a_{x-1}(\p),\underbrace{a_x(\r),0,\dots,0,b_y(\r)}_{\text{Stacks played on}},b_{y+1}(\p),\dots,b_{\ell}(\p))$$ 
with $y < x$, that is, play has occurred on at most  $\ell$ stacks, with $a_x(\r)>0$ and $b_y(\r)>0$. 
\end{lemma}

\begin{minipage}{.9\linewidth}
\begin{algorithm}[H]
\caption{$\delta$-Algorithm}
\label{alg:delta}
\begin{algorithmic}[1]
\Require $\p=(a_1,\dots,a_\ell,b_1,\dots,b_\ell)$ with $A=B$ and $s^*-m>0$
\State $\r \gets \p$
\State $m \gets \min(a_1,b_{\ell})$
\State Compute the quantities $s_i$ for $i=2,\dots,\ell+1$ 
\State Compute the quantities $m_j=\min_{2 \leq i \leq j}(s_i-m)$ for $j=2,\dots,\ell+1$
\State $\delta\gets m_{\ell+1}$
\State $x \gets \ell$,  \quad $y \gets 1$
\While{$\delta>1$ \textbf{and} $x\ge2$ \textbf{and} $y \le\ell-1$}
     \State $d \gets \min(a_x(\p),\, b_y(\p),\,  \lfloor \delta/2 \rfloor)$
    \If{$d>0$}
        \State $a_x(\r)\gets a_x(\p)-d$
        \State $b_y(\r)\gets b_y(\p)-d$
        \State $m_i \gets m_i-d$ for $i=2,\dots, x$
        \State $m_i \gets \min(m_i-d,m_{i-1})$ for $i=y+2,\dots, \ell+1$ \Comment{Separate loop}
        \State $\delta \gets m_{\ell+1}$  
    \EndIf
    \If{$a_x(\r)=0$}
        \State $x\gets x-1$
    \EndIf
    \If{$b_y(\r)=0$}
        \State $y\gets y+1$
    \EndIf
\EndWhile
\State \Return $\r$
\end{algorithmic}
\end{algorithm}
\end{minipage}
\vspace{0.5cm}

\begin{proof} The definition of $d$ ensures that the simultaneous reductions on two stacks keep $\delta\geq 0$
  (Remarks~\ref{rem:algs}.\ref{itm:impactind}~-~\ref{rem:algs}.\ref{itm:m_l+1}) and no stack falls below zero. In each step, the  $m_i$ values are reduced by at least $d$; those in the overlap of the affected sums are reduced by $2d$ and one needs to make sure to maintain the relationship that $m_j=\min_{2\leq i \leq j}(s_i-m)$. In the first loop (adjustment based on $a_x$), all $m_j$ get reduced by the same amount, so the defining relationship remains intact. In the second loop, those stacks in the overlap get reduced  by $2d$ overall, while the subsequent values of $m_j$ are only reduced by $d$. The assignment $m_i=\min(m_i-d,m_{i-1})$ takes care of the requirement that  $m_j=\min_{2\leq i \leq j}(s_i-m)$. 
  
  In each iteration, one of three things happens: 1) one of the stacks $a_x$ and $b_y$ is reduced to zero and its index is adjusted, creating a consecutive block of zeros, 2) $\delta$ is reduced to either zero or one, or 3) $a_x$ and $b_y$ are at distance $\ell-1$. In that case, there is no overlap between affected stacks and each $m_i$ is reduced by $d$ only, so $\delta$ is reduced to approximately half of its value. Also,  $a_x$ and $b_y$ remain positive, so no index is adjusted, and the reduction in the next step is on the same stacks.  Neither one of $a_x$ or $b_y$ can be reduced to zero before $\delta$ has been reduced  to either one or zero in subsequent steps because otherwise, there would be a sequence of $\ell-1$ consecutive zeros which is impossible by Remark~\ref{rem:algs}.\ref{itm:consec0}. Hence the algorithm cannot adjust more than $\ell$ adjacent stacks, the two end stacks  $a_x$ and $b_y$ must be non-negative, and $\delta \in\{0,1\}$ in the returned position. 
\end{proof}

\begin{example} \label{ex:Ddd2} We continue Example~\ref{ex:Ddd1}. Position $\r=(2,15,8,2,0,4,5,5,5,8)$ with $\delta=9$, $m=2$ and $(s_2,s_3,s_4,s_5,s_6)=(25,14,11,14,19)$ now becomes the input $\p$ to the  $\delta$-Algorithm.   For each iteration, we highlight the adjustments to the $m_i$ due to a reduction of $a_x$ in the row labeled with the value of $x$, and likewise for the effect of reduction in $b_y$. Values that do not get affected by the second adjustment remain the same, but are not repeated for ease of readability. Initially, $x=5$ and $y=1$. Since $d=0$, there is no change in values, except $x$ is reduced to $4$. (This iteration is not shown in the table.)

\begin{table}[!htb]
\centering
\begin{tabular}{|l||l|l|l|l|l||l||l|}
\hline
\rowcolor[HTML]{DAE8FC} 
$x,y$ & $m_2$ & $m_3$ & $m_4$ & $m_5$ & $m_6$ & $\delta$ & \multicolumn{1}{c|}{\cellcolor[HTML]{DAE8FC}$\r$} \\ \hline
      & $23$  & $12$  & $9$   & $9$   & $9$   & ${\bf 9}$      & $(2,15,8,2,0,4,5,5,5,8)$                          \\ \hline\hline
$x=4$ & $21$  & $10$  & $7$   &       &       &          &                                                   \\ \hline
$y=1$ &       & $8$   & $5$   & $5$   & $5$   &    ${\bf 5}$      &    $(2,15,8,{\bf 0},0,{\bf 2},5,5,5,8)$                                                \\ \hline\hline
$x=3$ & $19$  & $6$   &       &       &       &          &                                                   \\ \hline
$y=1$ &       & $4$   & $3$   & $3$   & $3$  & ${\bf 3}$      & $(2,15,{\bf 6},0,0,{\bf 0},5,5,5,8)$        \\ \hline \hline
$x=3$ & $18$  & $3$   &       &       &       &          &                                                   \\ \hline
$y=2$ &       &       & $2$   & $2$   & $2$  & ${\bf 2}$      & $(2,15,{\bf 5},0,0,0, {\bf 4},5,5,8)$            \\ \hline\hline
$x=3$ & $17$  & $2$   &       &       &       &          &                                                   \\ \hline
$y=2$ &       &       & $1$   & $1$   & $1$   & ${\bf 1}$      & $(2,15,{\bf 4},0,0,0, {\bf 3},5,5,8)$       \\ \hline
\end{tabular}
\end{table}

When $x=4$ and $y=1$, then $d=\min(a_x(\p),b_y(\p),\lfloor \delta/2 \rfloor)=\min(2,4,4)=2$, so $a_4(\r)=0$ and $b_1(\r)=2$. Now we recompute the $m_i$. Values affected by reduction of stack $a_4$ get reduced by  $2$. In the second set of adjustments, $m_3$ and $m_4$ both are reduced by $d=2$, while for $m_5$ and $m_6$, we have $m_i=\min(m_i-d,m_{i-1})=\min(7,5)=5$. Then $\delta=m_6=5$ (an overall reduction of $2d$), $x=3$, and the resulting position is $\r=(2,15,8,0,0,2,5,5,5,8)$. 
The next iteration has $d=\min(8,2,2)=2$, so $a_3(\r)=6$, $b_1(\r)=0$, $\delta=3$ (a reduction of $d$ only), $y=2$, and   $\r=(2,15,6,0,0,0,5,5,5,8)$. 
Now $x=3$ and $y=2$, so $d=\min(6,5,1)=1$. Thus $a_3(\r)=5$, $b_2(\r)=4$, $\delta = 2$ (reduction of $d$ only), and $\r=(2,15,5,0,0,0,4,5,5,8)$. Note that we continue to play on the same stacks, and obtain $d=\min(5,4,1)=1$, leading to $a_3(\r)=4$, $b_2(\r)=3$, and $\delta=1$. The algorithm ends with position $\r=(2,15,4,0,0,0,3,5,5,8)$. Note that there are $\ell-1$ consecutive zeros, surrounded by two stacks with positive stack height. The minimum window is $s_4=3$. We will explain in the next subsection how to make an additional adjustment that results in $\delta=0$. 
\end{example}

\subsection{Cheat Sheet}
Now that we have the necessary algorithms that create $A$-side and middle moves to  satisfy both the (SE) and (ME) conditions, we can summarize how to play. We distinguish by whether $m=\min(a_1,b_{\ell}) \geq s^*$ (Case 1) or not (Case 2). In the first case, the moves depend both on the relative sizes of $\Delta$ and $\delta$ and on the endpoint at which the minimum $m$ occurs.  In the second case, we will rely heavily on the three algorithms in combination with each other.   In both cases,  we justify that the indicated moves are legal and lead to a position in $S_{\ell}$.\\

 We start with Case 1. The table below indicates how to play  by listing the stacks of $\p'$ that have changed compared to $\p$.

\begin{table}[!htb]\centering
\begin{tabular}{|l|l||ll|}
\cline{3-4}
\multicolumn{2}{c|}{}  
     & \multicolumn{2}{c|}{\cellcolor[HTML]{DAE8FC}{\bf Case 1}: $\delta=m-s^*$}\\
\cline{3-4}
\multicolumn{2}{c|}{}                                             & \multicolumn{1}{l|}{ ~~~~~~~~~~$\delta \geq \Delta$}                  & ~~~~~~~~~~~~~~$\delta < \Delta$                                       \\  \hline
{\cellcolor[HTML]{DAE8FC}1.1 }& $m=a_1$                                                                     & \multicolumn{1}{l|}{\begin{tabular}[c]{@{}l@{}}$a_1'=a_1-\delta=s^*$\\ $b_{\ell}'=b_{\ell}-(\delta-\Delta)$\end{tabular}}                                                  & \begin{tabular}[c]{@{}l@{}}First set $a_1=s^*$, then apply \\ $\Delta$-Algorithm with $\Delta=A'-B$\end{tabular} \\ \hline
{\cellcolor[HTML]{DAE8FC}1.2} & \begin{tabular}[c]{@{}l@{}}$m=b_{\ell}$, \\ $a_1 \geq m+\Delta$\end{tabular}  & \multicolumn{2}{c|}{$a_1'=a_1-\Delta-\delta, \quad b_{\ell}'=b_{\ell}-\delta=s^*$}                            \\ \hline
{\cellcolor[HTML]{DAE8FC}1.3 }& \begin{tabular}[c]{@{}l@{}}$m=b_{\ell}$, \\ $m < a_1 < m+\Delta$\end{tabular} & \multicolumn{1}{l|}{\begin{tabular}[c]{@{}l@{}}$\Delta':=\Delta+m-a_1$\\ $a_1'=b_{\ell}-\delta=s^*$\\ $b_{\ell}'=b_{\ell}-(\delta-\Delta')$\end{tabular}} & \begin{tabular}[c]{@{}l@{}}First set $a_1=m$. \\ Then apply case 1.1.\end{tabular}                                                  \\ \hline
\end{tabular}
\caption{The winning move when $\delta=m-s^* \ge 0$}
 \label{tab:case1}
\end{table}

\noindent {\bf Case 1}: $\Delta=A-B>0$ and $\delta = m-s^*\geq 0$. 
\begin{enumerate}
    \item [1.1] When $\delta \geq \Delta$, then $a_1'=a_1-\delta \leq b_{\ell}-\delta \leq b_{\ell}'$, so (ME) holds. Also, $A'=A-\delta=B+\Delta -\delta=B'$, so (SE) holds. When $\delta < \Delta$, first reduce $a_1$ to $s^*$ to ensure (ME), then apply the $\Delta$-Algorithm to the resulting position. By Lemma~\ref{lem:A-side}, the additional reductions are on $A$-side stacks only and achieve (SE) while maintaining (ME).
    \item[1.2] Since $a_1 \geq m+\Delta$ we have that $a_1'\geq m-\delta=b_{\ell}'=s^*$, so (ME) holds. Also, $A'=A-\Delta-\delta=B-\delta=B'$ so (SE) holds. 
    \item[1.3] When $\delta \geq \Delta$, then  $m<a_1<m+\Delta$ implies that $0<\Delta'=\Delta+m-a_1<\Delta\leq \delta$. Thus the endpoint move on $a_1$ and $b_{\ell}$ to  $a_1'=s^*$ and $b_{\ell}'=b_{\ell}-(\delta-\Delta')\geq a_1'$  is legal and ensures (ME) as well as (SE) because $B'=B+\Delta' -\delta=A+m-a_1-\delta=A'$. On the other hand, when $\delta <\Delta$, then we first reduce $a_1$ to $m$. Since $a_1-m <\Delta$, we have $A'=A-a_1+m>A-\Delta=B$, that is, the intermediate position satisfies Case 1.1, so a move to $\p' \in S_{\ell}$ is available using either an endpoint move or an $A$-side move. 
\end{enumerate}

\noindent {\bf Case 2}: $\Delta=A-B \geq 0$ and $\delta = s^*-m >0$. In this case, $s^* >m$, so a pure endpoint move will not work as it does not change $s^*$. We may assume that $m>0$; otherwise, we play in the game \PN{2\ell-1}{\ell} for which we know the winning move from Theorem~\ref{thm:P-pos path}.  If $\Delta=0$, then apply the $\delta$-Algorithm to $\p$. By Lemma~\ref{lem:d}, the resulting position $\r$ either has $\delta = 0$ (in which case $\r \in S_{\ell}$) or $\delta = 1$, in which case we make the same adjustment as described in Case 2.1 below. If $\Delta> 0$, we employ the \Dd-Algorithm to fix one of the two inequalities. There are two possible outcomes:
\begin{enumerate}  
      \item[2.1] The \Dd-Algorithm returns a position $\r$ with $\Delta(\r)=0$ and $\delta(\r) >0$. In this case apply the $\delta$-Algorithm to $\r=(a_1,\dots,a_{j-1},r_j,0,\dots,0,b_{1},\dots,b_{\ell})$ with $j \geq 2$, that is, play has occurred on exactly $\ell-j+1 \leq \ell-1$ stacks from the $A$-side, with $r_j \ge 0$. The $\delta$-Algorithm has $d=0$ until $x=j$, and then reduces stacks $r_j(\r)$ and $b_1(\r)$. By Lemma~\ref{lem:d}, the algorithm returns a position where at most $\ell$ stacks have been played on and $\delta \in \{0,1\}$. If $\delta=0$, then the resulting position is in $S_{\ell}$ and we are done. If $\delta=1$, then we need to make a further adjustment. The block of stacks that were played on consists of consecutive zeros flanked by non-zero stacks $a_x$ and $b_y$ on both sides. Let  $t$ be the  index of the minimum window, that is,  $s_t=s^* =m+1$ and $s_i>s^*$ for $i<t$. The block of consecutive zeros must belong to $s_t$. By Remark~\ref{rem:algs}.\ref{itm:impactind}, $a_x$ is a summand in $s_i$ for $i=2,\dots,x$, while $b_y$ belongs to sums $s_i$ with $i=y+2,\dots,\ell+1$. Since $y<x$ by Lemma~\ref{lem:d}, we have that $y+2\leq x+1$, so at least one of $a_x$ and $b_y$ is a summand of the minimum sum $s_t$.  If exactly one of $a_x$ and $b_y$ belongs to the minimum window, then reduce both stacks by one. This will make $s^*=m$ while maintaining $A=B$.   If both $a_x$ and $b_y$ belong to the minimum window, then by definition of $t$  we have that $a_{t-1}>0$. Reduce $a_{t-1}$ and $b_y$ each by one to achieve (ME) and maintain (SE). The resulting position is in $S_{\ell}$.
      
       \item[2.2] The \Dd-Algorithm returns a position $\r$ with $\delta(\r)=0< \Delta(\r)$. If $a_1(\r) \geq m+\Delta(\r)$, set $a_1'(\r)=a_1(\r)-\Delta(\r)$. Otherwise, use the $\Delta$-Algorithm with $\r$. The respective resulting positions in these two cases have undergone only $A$-side moves and satisfy both (ME) and (SE).
\end{enumerate}

\begin{example} In Example~\ref{ex:Ddd2}, the $\delta$-Algorithm ended with position $\r=(2,15,4,\allowbreak 0,\allowbreak 0,\allowbreak 0, \allowbreak3,5,5,8)$ with $\delta=1$, $m=2$, $s^*=s_4=3$, and minimum window $(0,0,0,3)$. Here $a_x=4$ (outside the minimum window) and $b_y=3$ (inside the minimum window). Hence we reduce both stacks by $1$, resulting in option $\p'=(2,15,3,0,0,0,2,5,5,8)\in S_{\ell}$ with $A'=20=B'$ and $m'=2=(s^*)'$. 
\end{example}

\section{Conclusion and Future Work}\label{sec:conclusion}

We have solved  \NN{n}{k} games with $k \geq \lc n/2 \rc$. The remaining \ruleset{NecklaceNim} games with $k < \lc n/2 \rc$ have $G \defeq \{\{a, b\}, \{b, c\}, \{d\}\}$ as a
sub-game. To see this, consider the subspace of positions $\p \in \NN{n}{k} $ where all but stacks $p_{n-1}$, $p_n$, $p_1$, and $p_{\lfloor n/2\rfloor}$ and  have been reduced to zero. 
Since $k<\lc n/2\rc$, there are at least $k-1$ stacks strictly between the stacks $p_1$ and $p_{\lfloor n/2\rfloor}$ along the necklace, and likewise at least $k-1$ stacks strictly between stacks $p_{\lfloor n/2\rfloor}$ and $p_{n-1}$. Thus we can either play on just stack $p_{\lfloor n/2\rfloor}$ or on one of the pairs  $\{ p_{n-1},p_n\}$ or $\{p_n, p_1\}$. The game resulting after zero-reduction is $$G\defeq \{\{ p_{n-1},p_n\}, \{p_n, p_1\}, \{p_{\lfloor n/2\rfloor}\}\}\defeq \{\{a,b\}, \{b,c\}, \{d\}\}.$$ Thus solving $G$ is necessary to determine the \P-positions of the remaining \ruleset{NecklaceNim} games.

\begin{open} Solve the game \SN{4}{\{\{a,b\}, \{b,c\}, \{d\}\}}. 
\end{open}

\ruleset{NecklaceNim} games  arise in the subspace of positions of \ruleset{CircularNim} with $k-2$ consecutive zeros, just like \ruleset{PathNim} games arise when there is a string of $k-1$ consecutive zeros in \ruleset{CircularNim}.  This observation gives rise to more general \ruleset{NecklaceNim} games, where the ``clasp" is broader, corresponding to shorter strings of consecutive zeros in \ruleset{CircularNim}. 

\begin{definition} The game \ruleset{NecklaceNim} \NNg{n}{k}{c} with $n \geq k$ and $2 \leq c \leq \lf n/2 \rf+1$ is defined as the \ruleset{SetNim} game \SN{n}{\A} on the vertex set $V=\{1,\dots,n\}$ with $\A=\A_1 \cup \A_2$, where $\A_1$ represents the move set of \PN{n}{k} (the necklace) and $\A_2$ represents the move set of \PN{2\cdot(c-1)}{c} on the path 
\[(n-c+2) \,\text{---}\, (n-c+1) \,\text{---}\,  \dots \,\text{---}\, n \,\text{---}\, 1 \,\text{---}\, \dots \,\text{---}\, (c-1),\]
which represents the clasp of size $c$.
\end{definition}

Note that \NNg{n}{k}{2} $\cong$ \NN{n}{k}. Furthermore, for $s=k-c$, \NNg{n}{k}{c} arises as a sub-game of \CN{n+s}{k} on the space of positions that have  $s$ consecutive zeros.

\begin{open} Solve any of the \ruleset{NecklaceNim} \NNg{n}{k}{c} games.
\end{open}

\end{document}